\documentclass[11pt, reqno]{amsart}

\usepackage{amsmath, amsfonts, amssymb, amsbsy, bigstrut, graphicx, enumerate,  upref, longtable, comment}

\usepackage{pstricks}

\usepackage[breaklinks]{hyperref}

\usepackage[numbers, sort&compress]{natbib}

\newcommand{\diam}{\textup{diam}}

\newcommand{\leb}{\textup{Leb}}

\newcommand{\ep}{\epsilon}

\newcommand{\mc}{\mathcal{C}}

\newcommand{\mt}{\mathcal{T}}

\newtheorem{thm}{Theorem}[section]
\newtheorem{lmm}[thm]{Lemma}
\newtheorem{cor}[thm]{Corollary}

\theoremstyle{definition}

\newcommand{\cov}{\mathrm{Cov}}
\newcommand{\dm}{\mathcal{D}}

\newcommand{\ee}{\mathbb{E}}

\newcommand{\mf}{\mathcal{F}}

\newcommand{\pp}{\mathbb{P}}

\newcommand{\rr}{\mathbb{R}}

\newcommand{\uu}{\mathcal{U}}
\newcommand{\var}{\mathrm{Var}}

\newcommand{\zz}{\mathbb{Z}}
 
\newcommand{\vv}{\mathcal{V}}

\numberwithin{equation}{section}

\newcommand{\eq}[1]{\begin{align*} #1 \end{align*}}
\newcommand{\eeq}[1]{\begin{align} \begin{split} #1 \end{split} \end{align}}

\begin{document}
\title[Rigidity of the 3D hierarchical Coulomb gas]{Rigidity of the three-dimensional hierarchical Coulomb gas}
\author{Sourav Chatterjee}
\address{\newline Department of Statistics \newline Stanford University\newline Sequoia Hall, 390 Serra Mall \newline Stanford, CA 94305\newline \newline \textup{\tt souravc@stanford.edu}}
\thanks{Research partially supported by NSF grant DMS-1608249}
\keywords{Coulomb gas, interacting particles, rigidity, hyperuniformity}
\subjclass[2010]{60K35, 82B05}

\begin{abstract}
A random set of points in Euclidean space is called `rigid' or `hyperuniform' if  the number of points falling inside any given region has significantly smaller fluctuations than the corresponding number for a set of i.i.d.~random points. This phenomenon has received considerable attention in recent years, due to its appearance in random matrix theory, the theory of Coulomb gases and zeros of random analytic functions. However, most of the published results are in dimensions one and  two. This paper gives the first proof of hyperuniformity in a Coulomb type system in dimension three, known as the hierarchical Coulomb gas. This is a simplified version of the actual 3D Coulomb gas. The interaction potential in this model, inspired by Dyson's hierarchical model of the Ising ferromagnet, has a hierarchical structure and is locally an approximation of the Coulomb potential. Hyperuniformity is proved at both macroscopic and microscopic scales, with  upper and lower bounds for the order of fluctuations that match up to logarithmic factors. The fluctuations have cube-root behavior, in agreement with a well-known prediction for the 3D Coulomb gas. For completeness, analogous results are also proved for the 2D  hierarchical Coulomb gas and the 1D hierarchical log gas.
\end{abstract}

%The fluctuations follow a cube-root law, in agreement with a famous prediction of Jancovici, Lebowitz and Manificat for the three-dimensional Coulomb gas. 

\maketitle

%\tableofcontents

\vskip.5in
%\newpage
\section{Introduction and results}
\subsection{Interacting gases}
The probability density of $n$ independent and identically distributed points in $\rr^d$ can always be represented as
\[
\frac{1}{Z}\exp\biggl(-\beta \sum_{i=1}^n V(x_i)\biggr)
\]
where $V$ is some real-valued function on $\rr^d$, $\beta$ is some positive parameter, and $Z$ is the normalizing constant. %For example, $V(x)=|x|^2$ gives the centered Gaussian densities. 

Suppose that we want to introduce some interactions between the points. The simplest way to do that is to introduce a pairwise interaction term in the exponent; the new density is of the form
\[
\frac{1}{Z}\exp\biggl(-\beta \sum_{1\le i<j\le n} w(x_i,x_j) - \beta n \sum_{i=1}^n V(x_i)\biggr),
\]
where $w$ is a symmetric real-valued function on $\rr^d\times \rr^d$, known as the interaction potential. The factor $n$  is put in front of the second term to ensure that the two terms are of comparable size, which is necessary for ensuring that the system has nontrivial properties in the large $n$ limit. A particularly important type of interaction potentials are the Coulomb potentials, defined~as
\[
w(x,y) =
\begin{cases}
|x-y| &\text{ if } d=1,\\
-\log|x-y| &\text{ if } d=2,\\
|x-y|^{2-d} &\text{ if } d\ge 3,
\end{cases}
\]
where $|x-y|$ is the Euclidean distance between $x$ and $y$. 
With $w$ as above and $V(x)=|x|^2$, we get the so-called Coulomb gases.  

The 1D Coulomb gas is a very well-understood exactly solvable system, studied thoroughly by physicists~\cite{lenard61, lenard63, kunz74, dharetal17} and mathematicians~\cite{bl75, am80}. 
In higher dimensions, much less is known. For  $\beta=1$, the 2D Coulomb gas is an exactly solvable model due to its relationship with the Ginibre ensemble of random matrices~\cite{ginibre65}. The Ginibre ensemble has received widespread attention from mathematicians~\cite{girko84, ridervirag07, taovu08, borodinsinclair09, ahm11, ahm15, byy14a, byy14b, ghoshperes17}.  For general $\beta$, however, the 2D Coulomb gas has no representation as an exactly solvable model. Fortunately, a number of results are now known for the case of general $\beta$. Large deviation principles for the 2D Coulomb gas were proved in~\cite{bz98, petzhiai98, hardy12}, and extended to general dimensions in~\cite{chafaietal14, serfaty14}. Concentration inequalities were proved in~\cite{chafai16} and dynamical properties have been recently studied in~\cite{bolleyetal17}. The ground state in a related model was studied in~\cite{radin81}. Local properties have been studied in great depth in the recent papers~\cite{sandierserfaty15,leble15, bbny15, bbny16, lebleserfaty16}. 

In dimensions three and higher, very precise information about the normalizing constants has been obtained in \cite{rougerieserfaty16, lebleserfaty15}. For a comprehensive survey, see~\cite{serfaty15}. Further results are provable by the techniques of these papers but have not been written up yet, as I learned from Sylvia Serfaty in a personal communication.%, but an understanding of the finer  properties seems to be out of reach. 

Another widely studied example is the 1D log gas, where $d=1$, $w(x,y)=-\log|x-y|$ and $V(x)=x^2$. For $\beta=1,2$ and $4$, the log gases arise as eigenvalues of various random matrix ensembles and are exactly solvable. Precise fluctuation estimates for these special values of $\beta$ were obtained in~\cite{johansson98}. There is now considerable information available about other values of $\beta$ and more general $V$~\cite{shcherbina13, bey12, bey14, bey14b, beyy16, bekermanlodhia16}. Asymptotic series expansions for the normalizing constants were computed in~\cite{bg13, bg13b, bgk15}. Central limit theorems have been investigated in~\cite{bgg17, bls17, llw17}. For an introduction to log gases and their connections with random matrices, see~\cite{deift99, forrester10, agz10}. A recent survey is given in~\cite{bbny15}.

\subsection{Hyperuniformity}
If we have a collection of $n$ independent and identically distributed points in $\rr^d$, then the number of points that fall in a given set has fluctuations of order $n^{1/2}$ as $n\to\infty$. If a random point process has the property that this order of fluctuations is $o(n^{1/2})$, then it is called `rigid' or `hyperuniform'.  More generally, a point process is called hyperuniform if its empirical measure has smaller fluctuations than the empirical measure of a collection of i.i.d.~random points. In this paper I will use the terms `rigidity' and `hyperuniformity' interchangeably, but in general hyperuniformity is probably a more suitable term for the phenomenon described above, since rigidity has also been used to mean other things in the literature. %Physicists prefer to use the term `hyperuniform', in the sense that such point processes are `more uniform than i.i.d.~uniform'.

Sometimes point processes are very rigid, such as eigenvalues of various random matrix ensembles, for which the order of fluctuations may be as small as $O(\sqrt{\log n})$ or even $O(1)$ if one considers integrals of the empirical measure with respect to smooth functions.  Rigidity/hyperuniformity  has been established for many processes in dimensions one and two. For example, rigidity of the eigenvalues of random unitary matrices was proved  in~\cite{diaconisevans01, wieand02}.  Rigidity of eigenvalues in the standard  hermitian random matrix ensembles follow from the results of~\cite{costinlebowitz95, pastur06, bey12, bey14, bey14b, ghosh15, taovu13}. Rigidity of eigenvalues of non-hermitian random matrices has been studied  in~\cite{borodinsinclair09,byy14a, byy14b,taovu15, ghoshperes17}. Another class of 2D processes that exhibit hyperuniformity are zeros of random analytic functions.  This has been investigated in~\cite{nsv07, nsv08, houghetal09, nazarovsodin11, ghosh16, gz16, gl17, ghoshperes17}. In recent work, rigidity of the 2D Coulomb gas has been established in~\cite{bbny15, bbny16, lebleserfaty16}.

However, no such results for interacting particle systems of the above kind are known in dimensions three and higher. The only random point process which has been shown to be hyperuniform in any dimension $d\ge 3$, as far as I know, is the point process obtained by giving i.i.d.~random perturbations to the vertices of $\zz^d$. This is a recent result~\cite{peressly14}, improving on an earlier work in $d\le 2$~\cite{holroydsoo13}. The notion of rigidity in these papers is somewhat different than hyperuniformity. For interacting systems such as Coulomb gases, the detailed information about the normalizing constants obtained in~\cite{serfaty15, rougerieserfaty16} provide some control on the order of fluctuations in $d\ge 3$, but do not establish that the order of fluctuations is smaller than $n^{1/2}$. There is a remark in~\cite{lebleserfaty16} that the 2D techniques of that paper can be extended to higher dimensions for proving rigidity of integrals of smooth functions with respect to the empirical measure of a Coulomb gas, but the details have not yet been written up. 

A class of processes that are related to Coulomb gases in dimension one but not in higher dimensions, are the so-called orthogonal polynomial ensembles~(see \cite{konig05} for a survey). These are generalizations of the 1D determinantal point processes arising in random matrix theory, and have nice mathematical structures that allow various exact calculations. A general central limit theorem for orthogonal polynomial ensembles was proved in~\cite{soshnikov02}. Rigidity/hyperuniformity for orthogonal polynomial ensembles (beyond random matrix eigenvalues) have been investigated in recent years, for example in \cite{berman14, jl15, breuerduits16, breuerduits17} in dimensions one and two, and \cite{bardenethardy16} in dimension three.

There is a considerable amount of work by physicists  on hyperuniformity. For example, \cite{martinyalcin80} and \cite{martin88} give physics proofs of hyperuniformity in 3D Coulomb systems. A non-rigorous computation of covariances in Coulomb systems in all dimensions greater than one was given in~\cite{lebowitz83}. More recently, a physics proof of hyperuniformity of free fermions at zero temperature (a certain kind of determinantal point process)    was given in~\cite{castin} in $d\le3$, based  on  an asymptotic formula  for the variance of the number of points falling in a given region. This formula was later extended to arbitrary dimensions in~\cite{torquatoetal08}. Similar formulas have been very recently obtained for the 1D log gas (with general $\beta$ and special $V$) in~ \cite{marinoetal14, marinoetal16}. For an extensive list of references to the physics literature, see the recent survey~\cite{gl17b}.

%Philosophically speaking, rigidity arises whenever the points in a random point process behave as if they are `sticking' close to some prescribed values. One can see how this works for random perturbations of the integer lattice, and also to some extent for random eigenvalues (where one can talk about the concentration of the $k^\textup{th}$ largest eigenvalue around its mean), but it is hard to see what is going on for permutation invariant systems such as Coulomb  gases.

\subsection{The hierarchical Coulomb gas model}
In this paper, we consider a model of an interacting gas of $n$ particles in the 3D unit cube $[0,1]^3$, which have joint probability density
\eeq{
\frac{1}{Z(n,\beta)}\exp\biggl(-\beta \sum_{1\le i<j\le n}w(x_i,x_j)\biggr),\label{densdef}
}
where $w(x,y)$ is a symmetric potential that behaves like the Coulomb potential $|x-y|^{-1}$ at short distances, and $Z(n,\beta)$ is the normalizing constant.  The potential $w$ is  defined as follows.

The unit cube in $\rr^3$ can be partitioned into $8$ sub-cubes of side-length $1/2$. Each of these sub-cubes can be further partitioned into $8$ sub-cubes of side-length $1/4$, and so on, generating a tree of dyadic sub-cubes. For any two distinct points $x$ and $y$ in the unit cube, let $w(x,y)=2^k$, where $k$ is the smallest number such that $x$ and $y$ belong to distinct dyadic sub-cubes of side-length $2^{-k}$. There may be some ambiguity about points on the boundaries of the cubes, but since they form a set of measure zero, they  do not matter.  This $w$ is our potential, which defines our point process through the density \eqref{densdef}. Note that $w$ is symmetric but not translation invariant. 

For typical $x$ and $y$ which are close together, $w(x,y)$ behaves like a multiple of the Coulomb potential $|x-y|^{-1}$.  Indeed, it is not hard to prove that there is a constant $C$ such that for all $x$ and $y$,
\[
w(x,y)\le \frac{C}{|x-y|}.
\]
Conversely, there is a constant $c>0$ such that for any $0<\delta<1$, the average value of $w(x,y)$ over all pairs $(x,y)$ with $|x-y|=\delta$ is bounded below by $c/\delta$. 

Replacing the Euclidean distance by a hierarchical distance as above is a famous idea of \citet{dyson53, dyson69}, who formulated and analyzed a hierarchical version of the 1D Ising model with long range interactions. This is now known as `Dyson's hierarchical model'. Dyson's work has inspired a large body of literature on hierarchical models over the years, and is still an active area of research. The model proposed above is sometimes called the `hierarchical Coulomb gas'. The 2D hierarchical Coulomb gas has received considerable attention in the mathematical physics literature~\cite{benfattoetal86, marchettiperez89, dimock90, kappeleretal91, benfattorenn92, guidimarchetti01}. However, not much is known about this model in dimensions three and higher.

Just as the Coulomb potential is the Green's function for Brownian motion, the potential $w$ can also be realized as the Green's function of a certain continuous time random walk on the unit cube, following a method developed in~\cite{bendikovetal12, bendikovetal} for constructing Markov semigroups on ultrametric spaces. More generally, the prescription given in~\cite{bendikovetal12, bendikovetal} can be used for a large class of  hierarchical potentials arising from Dyson-type constructions. %, and that $1/w$ has the interesting feature that it is an ultrametric.

%Just as the Coulomb potential is the inverse of the Laplacian, which is the generator of Brownian motion, the potential $w$ is also the inverse of the generator of a Markov process. 

The chief reason why the hierarchical structure of the potential helps in the analysis is that it does an automatic `coarse-graining' of the interactions. The total interaction between the particles in two disjoint dyadic cubes is determined solely by the numbers of particles in those cubes, rather than their exact locations. 

One of our main results, stated in the next subsection, is that if $U$ is a nonempty open subset of the unit cube with a nicely behaved boundary, then the number of points falling in $U$ has fluctuations of order at most $n^{1/3}\sqrt{\log n}$, thereby establishing the hyperuniformity of our point process. This is matched up to a logarithmic factor by a  lower bound of order $n^{1/3}$. We also establish microscopic hyperuniformity in a local neighborhood of any given point. Finally, the analogous results in dimensions one and two are established for the sake of completeness.

\subsection{Results in 3D}\label{results}
Take any $d\ge 1$, and let $U$ be a nonempty open subset of $\rr^d$. Let $\partial U$ be the boundary of $U$. For each $\ep>0$, let $\partial U_\ep$ be the set of all points that are at distance $\le \ep$ from $\partial U$.  Let $\diam(U)$ denote the diameter of $U$. We will say that the boundary of $U$ is {\it regular} if there is some constant $C$ such that for all $0< \ep\le \diam(U)$,
\eeq{
\textup{Leb}(\partial U_\ep) \le C\ep,\label{regeq}
}
where $\textup{Leb}$ stands for Lebesgue measure.

Now let $d=3$, and let $U$ be a nonempty open subset of $[0,1]^3$ whose boundary is regular in the sense defined above. Take any $n\ge 2$ and $\beta>0$, and consider an interacting gas of $n$ particles behaving according to the model defined above.  Let $N(U)$ be the number of particles that fall in~$U$. 
Our first theorem says that the gas is macroscopically hyperuniform in the sense that $N(U)$ has fluctuations of order at most $n^{1/3}\sqrt{\log n}$, instead of $n^{1/2}$ as would be the case for a gas of i.i.d.~particles.
\begin{thm}[Macroscopic hyperuniformity in 3D]\label{macthm}
Let $U$ and $N(U)$ be as above.  Then 
\[
\ee(N(U)) = \leb(U)n
\]
and 
\eq{
\var(N(U))&\le C(U,\beta)n^{2/3}\log n,
}
where  $C(U,\beta)$ is a constant that depends only on $U$ and $\beta$.
\end{thm}
%It is not clear whether $n^{1/3}\sqrt{\log n}$ is indeed to correct order of fluctuations. Incidentally, the model can also be defined in a similar way in dimension two, and rough calculations indicate that a similar proof technique will lead to an upper bound of order $\sqrt{\log n}$, which may well be the correct order in dimension two. 

The next theorem shows that when $\partial U$ is smooth, $n^{1/3}$ is actually the correct order of fluctuations of $N(U)$, up to possible logarithmic corrections. %This is somewhat surprising given the much tighter concentration results in one-dimensional models~\cite{costinlebowitz95, wieand02, diaconisevans01, bey12, bey14, taovu13}.
\begin{thm}[Lower bound in 3D]\label{lowthm}
Let $U$ be a nonempty connected open subset of $[0,1]^3$ whose boundary is a smooth,  closed, orientable surface. Let $N(U)$ be as in Theorem~\ref{macthm}. Then $N(U)$ has fluctuations of order at least $n^{1/3}$, in the sense that there are three constants $n_0\ge 1$, $c_1>0$ and $c_2<1$, depending only on $U$ and $\beta$, such that for any $n\ge n_0$ and any $-\infty<a\le b<\infty$ with $b-a\le c_1n^{1/3}$, we have $\pp(a\le N(U)\le b) \le c_2$.
\end{thm}
Incidentally, the $n^{1/3}$ order of fluctuations matches a well-known prediction from physics~\cite{lebowitz83, martinyalcin80, jancovicietal93} for the 3D Coulomb gas model (see also~\cite{nsv08}). The $1/3$ exponent is also reminiscent of a famous classical result~\cite{beck87} about irregularities in distributions of arbitrary sequences of points in Euclidean space.

Let us now turn our attention to hyperuniformity in the microscopic scale. Take any point $x\in (0,1)^3$. Blow up the neighborhood of $x$ by a factor of $n^{1/3}$ by applying the blow-up map $y\mapsto n^{1/3}(y-x)$ to the points in our interacting gas.  Since the original process had an expected density of $n$ particles per unit volume, the new process has an expected density of one particle per unit volume. Studying the blown up process is  the standard way of investigating the local behavior of interacting gases~\cite{serfaty15}.

Let $U$ be a nonempty  open subset of $\rr^3$ whose boundary is regular in the sense defined above. For each $\lambda >0$, let $\lambda U$ denote the set $\{\lambda y:y\in U\}$, and let $N_x(\lambda U)$ be the number of points from the blown up process that land in $\lambda U$. The following theorem shows that for $\lambda \gg 1$, $N_x(\lambda U)$ has fluctuations of order at most $\lambda \sqrt{\log \lambda}$. This is smaller than $\lambda^{3/2}$, the corresponding order of fluctuations for a Poisson point process. This proves the hyperuniformity of our interacting gas at the microscopic scale. 
\begin{thm}[Microscopic hyperuniformity in 3D]\label{micthm}
Let $U$ and  $N_x(\lambda U)$ be as above.  Then for any $\lambda$ such that $\diam(\lambda U)\ge 1$,
\[
\lim_{n\to \infty} \ee(N_x(\lambda U)) = \leb(\lambda U) = \lambda^3\leb(U),
\]
and 
\eq{
\limsup_{n\to \infty} \var(N_x(\lambda U))&\le C(U, \beta)\lambda^2\log (4\lambda\diam(U)),
}
where $C(U, \beta)$ is a constant that depends only on $U$ and $\beta$.
\end{thm}
Theorems \ref{macthm} and \ref{micthm} are both special cases of a more general theorem (Theorem \ref{mainthm} in Section \ref{proof3d}), which gives hyperuniformity at all scales.  

Finally, let us consider linear statistics. Any function $f:[0,1]^3\to \rr$ defines a linear statistic
\eeq{
X(f) := \sum_{i=1}^n f(X_i),\label{lindef}
}
where $X_1,\ldots, X_n$ is a realization of our point process. In particular, $N(U)$ is a linear statistic, with $f$ being the indicator function of $U$. We have the following two theorems about fluctuations of linear statistics when $f$ is continuous. The results are not as definitive as the other results of this section, since the upper and lower bounds do not match. 

If $f$ is Lipschitz, we get the following slight improvement of the bound given in Theorem \ref{macthm}.
\begin{thm}[Upper bound for linear statistics in 3D]\label{linearthm}
Suppose that $f:[0,1]^3\to \rr$ is a Lipschitz function with Lipschitz constant $L$. Let $X_1,\ldots,X_n$ be a realization of points from our model in dimension two. Let $X(f)$ be the linear statistic defined in \eqref{lindef}. Then
\[
\var(X(f))\le C(\beta)L^2n^{2/3},
\]
where $C(\beta)$ is a constant that depends only on $\beta$. 
\end{thm}
The next theorem gives a lower bound   of order $n^{1/6}$ on the order of fluctuations of $X(f)$ when $f$ is a non-constant linear function.  This does not match the upper bound from Theorem~\ref{linearthm}, but is nonetheless growing polynomially in $n$, deviating from the $O(1)$ rate for smooth linear statistics in dimensions one and two~\cite{costinlebowitz95, diaconisevans01, wieand02, bey12, bey14, bey14b, johansson98, bbny15, bbny16, lebleserfaty16, taovu13, byy14a, byy14b}. 
\begin{thm}[Lower bound for linear statistics in 3D]\label{linlowthm}
Let $f:[0,1]^3\to \rr$ be a non-constant linear function, and let $X(f)$ be as in \eqref{lindef}. Then $X(f)$ has fluctuations of order at least $n^{1/6}$, in the sense that there are three constants $n_0\ge 1$, $c_1>0$ and $c_2<1$, depending only on $U$ and $\beta$, such that for any $n\ge n_0$ and any $-\infty<a\le b<\infty$ with $b-a\le c_1n^{1/6}$, we have $\pp(a\le X(f)\le b) \le c_2$.
\end{thm}
It is not clear whether $n^{1/3}$ or $n^{1/6}$ is the correct order of fluctuations for smooth linear statistics. Theorem \ref{lowthm} does not provide any strong evidence in favor of $n^{1/3}$, because, as we will see later for the 2D hierarchical Coulomb gas, linear statistics of smooth functions may have much smaller fluctuations than linear statistics of indicator functions. However, there is a  recent result~\cite{bardenethardy16} which shows that $n^{1/3}$ is the correct order of fluctuations for smooth linear statistics of a 3D orthogonal polynomial ensemble. Although orthogonal polynomial ensembles are not related to Coulomb type systems in dimension three, this gives some support in favor of $n^{1/3}$.

\subsection{Results in 2D and 1D}\label{results2d}
In dimension two, we will modify $w$ to mimic the logarithmic potential of the 2D Coulomb gas. This is done by declaring $w(x,y) = $ the minimum $k$ such that $x$ and $y$ belong to distinct dyadic sub-squares of $[0,1]^2$ of side-length $2^{-k}$. We will use the same formula in dimension one as well (with dyadic intervals instead of squares), so that $w$ mimics the logarithmic potential of 1D log gases.  With these modifications, we have the following analogs of Theorem \ref{macthm}. With $N(U)$ as in Theorem \ref{macthm}, it says that $N(U)$ has fluctuations of order at most $n^{1/4}\log n$ in dimension two, and $\log n$ in dimension one. 
\begin{thm}[Macroscopic hyperuniformity in 2D and 1D]\label{macthm2d}
Consider the model defined above in dimension $d=1$ or $2$. Let $U$ and $N(U)$ be as in Theorem~\ref{macthm}. Then
\[
\ee(N(U))=\leb(U)n
\]
and 
\[
\var(N(U))\le C(U,\beta) n^{(d-1)/d} (\log n)^2, 
\]
where $C(U,\beta)$ is a constant that depends only on $U$ and $\beta$.
\end{thm}
The following theorem shows that in dimension two, $N(U)$ has fluctuations of order at least $n^{1/4}$, matching the above upper bound up to a logarithmic factor.
\begin{thm}[Lower bound in 2D]\label{lowthm2d}
Let $U$ be a nonempty connected open subset of $[0,1]^2$ whose boundary is a simple, smooth,  closed curve. Let $N(U)$ be as in Theorem~\ref{macthm2d}. Then $N(U)$ has fluctuations of order at least $n^{1/4}$, in the sense that there are three  constants $n_0\ge 1$, $c_1>0$ and $c_2<1$, depending only on $U$ and $\beta$, such that for any $n\ge n_0$ and any $-\infty<a\le b<\infty$ with $b-a\le c_1n^{1/4}$, we have~$\pp(a\le N(U)\le b) \le c_2$.
\end{thm}
Like the $n^{1/3}$ rate in the 3D case, the $n^{1/4}$ rate  was also predicted in the physics literature \cite{lebowitz83, martinyalcin80, jancovicietal93} for the 2D Coulomb gas. The $n^{1/4}$ fluctuation in the special case of $\beta=1$ in the 2D Coulomb gas (corresponding to the exactly solvable Ginibre ensemble) can be established by standard techniques, as I learned from Paul Bourgade in a personal communication.

We also have the following analog of Theorem \ref{micthm}. With $N_x(\lambda U)$ as in Theorem \ref{micthm}, it shows that for $\lambda\gg1$, $N_x(\lambda U)$ has fluctuations of order at most $\lambda^{1/2}\log \lambda$ in dimension two, and $\log \lambda$ in dimension one. 
\begin{thm}[Microscopic hyperuniformity in 2D and 1D]\label{micthm2d}
Consider the model defined above in dimension $d=1$ or $2$. Let $U$ and  $N_x(\lambda U)$ be as in Theorem~\ref{micthm}.  Then for any $\lambda$ such that $\diam(\lambda U)\ge 1$,
\[
\lim_{n\to \infty} \ee(N_x(\lambda U)) = \leb(\lambda U) = \lambda^d\leb(U),
\]
and 
\eq{
\limsup_{n\to \infty} \var(N_x(\lambda U))&\le  C(U, \beta)\lambda^{d-1} (\log(7\lambda^d\diam(U)^d))^2,
}
where $C(U, \beta)$ is a constant that depends only on $U$ and $\beta$.
\end{thm}
As before, Theorems \ref{macthm2d} and \ref{micthm2d} are special cases of a more general theorem (Theorem \ref{mainthm2d} in Section \ref{proof2d}) that gives hyperuniformity at all scales. 

Finally, let us consider linear statistics. It has been proved recently in~\cite{bbny15, bbny16, lebleserfaty16} that for the 2D Coulomb gas, linear statistics of smooth functions have $O(1)$ fluctuations. For Lipschitz $f$, the following theorem shows that for our model in dimension two, the fluctuations of $X(f)$ are at most of order $(\log n)^{3/2}$ instead of $n^{1/4}$.  Unlike Theorem \ref{linearthm}, this is a big improvement of the bound from Theorem \ref{macthm2d}, and is within a logarithmic factor of the $O(1)$ bound from~\cite{bbny15, bbny16, lebleserfaty16}.
\begin{thm}[Upper bound for linear statistics in 2D and 1D]\label{linearthm2d}
Let $d=1$ or $2$. Suppose that $f:[0,1]^d\to \rr$ is a Lipschitz function with Lipschitz constant $L$. Let $X_1,\ldots,X_n$ be a realization of points from our model in dimension $d$, and let $X(f)$ be the linear statistic defined in \eqref{lindef}. Then
\[
\var(X(f))\le C(\beta)L^2(\log n)^{d+1},
\]
where $C(\beta)$ is a constant that depends only on $\beta$. 
\end{thm}

\section{Proofs in 3D}
The rest of this paper is devoted to proofs. In this section, we will prove the theorems of Section \ref{results}. 
\subsection{Notation}
It is helpful to define some precise notations and terminologies. For a slight technical convenience, we will replace the unit cube by the half-open unit cube $[0,1)^3$. Clearly, this will not alter the conclusions.

A dyadic sub-interval of the half-open unit interval $[0,1)$ is an interval of the form $[i2^{-k}, (i+1)2^{-k})$, where $k\ge 0$ and $0\le i\le 2^k-1$. A dyadic sub-cube of the half-open unit cube $[0,1)^3$ is a sub-cube of the form $I_1\times I_2\times I_3$, where $I_1$, $I_2$ and $I_3$ are dyadic sub-intervals of $[0,1)$ of equal length. Let $\dm_k$ be the set of all dyadic sub-cubes of $[0,1)^3$ of side-length $2^{-k}$, and let 
\[
\dm := \bigcup_{k=0}^\infty \dm_k
\]
be the set of all dyadic sub-cubes of $[0,1)^3$. Then $\dm$ has a natural tree structure, with each node having $8$ children. We will freely use the terms `child', `parent', `ancestor' and `descendant' with respect to this tree. 

For any two distinct points $x,y\in [0,1)^3$, let $k(x,y)$ be the smallest $k$ such that $x$ and $y$ belong to distinct elements of $\dm_k$. Then our potential $w$ is the function $w(x,y) = 2^{k(x,y)}$.  For $x=y$, let $w(x,y)=\infty$. 
% For typical $x$ and $y$, $k(x,y)$ is equal to $\log_2|x-y|$, up to a bounded error. Here $|x-y|$ is the Euclidean distance between $x$ and $y$, and $\log_2$ denotes logarithm to the base~$2$. From this it is not difficult to see that for almost all $x$, $w(x,y)$ behaves like the Coulomb potential $|x-y|^{-1}$ when $y$ is close to $x$. %Moreover, $w$ is bounded above by a constant multiple of the Coulomb potential.

%Our model is the following: 
For each $n\ge 2$, let $\Sigma_n$ be the set of all $n$-tuples of points from $[0,1)^3$. Define the energy of a configuration $(x_1,\ldots, x_n)\in \Sigma_n$~as 
\[
H_n(x_1,\ldots, x_n) := \sum_{1\le i<j\le n} w(x_i, x_j).
\]
For $\beta>0$, let $\mu_{n,\beta}$ be the probability measure on $\Sigma_n$ that has density
\[
\frac{1}{Z(n,\beta)}e^{-\beta H_n(x_1,\ldots, x_n)}
\]
with respect to Lebesgue measure on $\Sigma_n$, where $Z(n,\beta)$ is the normalizing constant. The measure $\mu_{n,\beta}$ defines our model of an interacting gas at inverse temperature $\beta$. 

For certain technical reasons, we will also define the model for $n=0$ and $n=1$. When $n=0$, there are no points. When $n=1$, there is one point which is uniformly distributed in the cube. We will let $Z(0,\beta)=Z(1,\beta)=1$ for any $\beta$.

\subsection{Preliminary calculations}\label{prelimsec}
In the following, all integrals are over $[0,1)^3$ and all double integrals are over $[0,1)^3\times [0,1)^3$, unless otherwise specified. 
\begin{lmm}\label{wlmm}
For each $x\in [0,1)^3$,
\[
\int w(x,y) \, dy = \frac{7}{3}.
\]
Consequently,
\[
\iint w(x,y) \, dx\, dy = \frac{7}{3}. 
\]
\end{lmm}
\begin{proof}
Take any $x$. For each $k$, let $D_k$ be the element of $\dm_k$ that contains $x$. It is easy to see that the set of all $y$ with $w(x,y)=2^k$ is exactly the union of all members of $\dm_k$ that are contained in $D_{k-1}$, except the one that contains~$x$. The Lebesgue measure of this set is $8^{-k}\cdot 7$. Thus,
\[
\int w(x,y)\, dy = 7\sum_{k=1}^\infty 2^k8^{-k} = \frac{7}{3}.
\]
The second assertion is obvious from the first.
\end{proof}
Let us now investigate energy-minimizing configurations of finite size. 
Henceforth, $L_n$ will denote the minimum possible energy of a configuration of $n$ points. The following result gives upper and lower bounds for $L_n$. 
\begin{thm}\label{lnlowthm}
There is a positive constant $C_1$ such that for each $n\ge 2$,
\[
{n\choose 2}\frac{7}{3} - C_1n^{4/3}\le L_n \le {n\choose 2}\frac{7}{3}.
\]
\end{thm}
\begin{proof}
Let $Y_1,\ldots, Y_n$ be i.i.d.~uniform random points from $[0,1)^3$. Then by symmetry, 
\eq{
L_n &\le \ee(H_n(Y_1,\ldots, Y_n)) = \sum_{1\le i<j\le n} \ee(w(Y_i, Y_j)) = {n\choose 2} \ee(w(Y_1,Y_2)).
}
By Lemma \ref{wlmm}, $\ee(w(Y_1,Y_2)) = 7/3$. This proves the upper bound. For the lower bound, let $k$ be an integer such that
\[
n^{-1/3}\le 2^{-k}\le 2n^{-1/3}. 
\]
Take any configuration of $n$ points. For each $D\in \dm$, let $n_D$ be the number of points in $D$. Summing up the contributions to the energy from each cube, it is not difficult to see that
\eq{
H_n(x_1,\ldots, x_n) &= \sum_{j=1}^\infty \sum_{D\in \dm_j} 2^j{n_D\choose 2} + 2{n\choose 2}\ge \sum_{j=1}^k \sum_{D\in \dm_j} 2^j{n_D\choose 2} + 2{n\choose 2}\\
&= \sum_{j=1}^k \sum_{D\in \dm_j} 2^{j-1}n_D^2 -  \sum_{j=1}^k 2^{j-1}n + 2{n\choose 2}.
}
By the Cauchy--Schwarz inequality, for each $j$,
\eq{
\sum_{D\in\dm_j} n_D^2 &\ge \frac{1}{|\dm_j|}\biggl(\sum_{D\in \dm_j} n_D\biggr)^2 = \frac{n^2}{8^j}. 
}
Thus,
\eq{
H_n(x_1,\ldots, x_n) &\ge \frac{n^2}{2}\sum_{j=1}^k 4^{-j} - n^{4/3} + 2{n\choose 2}= \frac{n^2}{6}(1-4^{-k})- n^{4/3} + 2{n\choose 2}\\
&\ge \frac{n^2}{6}(1-4n^{-2/3}) -n^{4/3} + 2{n\choose 2}.
}
Since this lower bound holds for any configuration of $n$ points, this completes the proof.
\end{proof}

\subsection{Estimates for the partition function}
The following lemma gives important information about the ratio $Z(n+1,\beta)/Z(n,\beta)$. Theorem \ref{lnlowthm} is a crucial ingredient in the proof of this lemma. Recall that $Z(0,\beta)=Z(1,\beta)=1$. For a measurable function $f: \Sigma_n\to\rr$, we will denote its expected value under $\mu_{n,\beta}$ by $\mu_{n,\beta}(f)$.  %The expectation under the uniform distribution on $\Sigma_n$ will be denoted by $\ee(f)$. 
\begin{lmm}\label{znlmm}
There is a constant $C_2$ such that for any $n\ge 0$ and $\beta>0$,
\[
e^{-7\beta n/3}\le \frac{Z(n+1, \beta)}{Z(n, \beta)}\le e^{-7\beta n/3 + C_2\beta  n^{1/3}}. 
\]
\end{lmm}
\begin{proof}
First suppose that $n\ge 2$. For $x_1,\ldots, x_n, x_{n+1}\in [0,1)^3$, let
\[
f_n(x_1,\ldots,x_n, x_{n+1}) := \sum_{i=1}^n w(x_i, x_{n+1}),
\]
so that 
\eq{
H_{n+1}(x_1,\ldots,x_{n+1}) &= f_n(x_1,\ldots,x_n, x_{n+1}) + H_n(x_1,\ldots,x_n). 
}
By the above representation and Jensen's inequality,
\eq{
&\frac{Z(n+1, \beta)}{Z(n,\beta)} = \iint e^{-\beta f_n(x_1,\ldots,x_n, x_{n+1})} \, dx_{n+1}\, d\mu_{n,\beta}(x_1,\ldots,x_n) \\
&\ge \exp\biggl(-\beta\iint f_n(x_1,\ldots,x_n, x_{n+1}) \,  dx_{n+1}\, d\mu_{n,\beta}(x_1,\ldots,x_n) \biggr).
}
But by Lemma \ref{wlmm},
\eq{
&\iint f_n(x_1,\ldots,x_n,x_{n+1}) \, dx_{n+1}\, d\mu_{n,\beta}(x_1,\ldots,x_n)  \\
&=\sum_{i=1}^n \iint w(x_i, x_{n+1}) \, dx_{n+1}\, d\mu_{n,\beta}(x_1,\ldots,x_n) \\
&= \sum_{i=1}^n \int \frac{7}{3}\, d\mu_{n,\beta}(x_1,\ldots,x_n) = \frac{7n}{3}.
}
This gives the desired lower bound.  Next, note that 
\eq{
\frac{Z(n,\beta)}{Z(n+1,\beta)} &= \mu_{n+1,\beta}(e^{\beta f_n(x_1,\ldots, x_n, x_{n+1})}). 
}
Therefore by Jensen's inequality and the invariance of $\mu_{n+1,\beta}$ under permutations of coordinates,
\eq{
\frac{Z(n,\beta)}{Z(n+1,\beta)} &\ge \exp(\beta \mu_{n+1,\beta}(f(x_1,\ldots, x_{n+1}))= \exp(\beta n \mu_{n+1,\beta} (w(x_1,x_{n+1})))\\
&= \exp\biggl(\frac{\beta n}{{n+1\choose 2}} \sum_{1\le i<j\le n+1} \mu_{n+1,\beta} (w(x_i,x_j))\biggr)  \\
&= \exp\biggl(\frac{\beta n}{{n+1\choose 2}} \mu_{n+1,\beta}(H_{n+1}(x_1,\ldots,x_{n+1}))\biggr).
}
But by Theorem \ref{lnlowthm}, 
\eq{
\mu_{n+1,\beta}(H_{n+1}(x_1,\ldots,x_{n+1})) &\ge L_{n+1} \ge \frac{7}{3}{n+1\choose 2} - C_1 (n+1)^{4/3}. 
}
This gives the required upper bound and completes the proof of the lemma for $n\ge 2$. When $n=0$, the bounds hold trivially. When $n=1$, the lower bound follows from an application of Jensen's inequality and Lemma \ref{wlmm}. The upper bound can be forced to hold for $n=1$ by choosing $C_2$ sufficiently large. 
\end{proof}
Lemma \ref{znlmm} is iterated to obtain the following corollary.
\begin{cor}\label{zncor}
For any $n\ge 0$, $\beta>0$, and any $k\ge -n$,
\[
\frac{Z(n+k, \beta)}{Z(n,\beta)} \le \exp\biggl(-\frac{7\beta nk}{3} -\frac{7\beta k(k-1)}{6} + C_2 \beta |k| (n+|k|)^{1/3}\biggr),
\]
where $C_2$ is the constant from Lemma \ref{znlmm}.
\end{cor}
\begin{proof}
First suppose that $k\ge 0$. By the upper bound from Lemma~\ref{znlmm},
\eq{
\frac{Z(n+k, \beta)}{Z(n,\beta)} &= \prod_{i=0}^{k-1}\frac{Z(n+i+1,\beta)}{Z(n+i,\beta)}\\
&\le  \prod_{i=0}^{k-1}\exp\biggl(-\frac{7\beta(n+i)}{3} + C_2\beta (n+i)^{1/3}\biggr)\\
&\le \exp\biggl(-\frac{7\beta nk}{3} -\frac{7\beta k(k-1)}{6} + C_2 \beta k (n+k)^{1/3}\biggr). 
}
Next, suppose that $k<0$. Let $l= |k|$. Then by the lower bound from Lemma \ref{znlmm}, 
\eq{
\frac{Z(n+k, \beta)}{Z(n,\beta)} &= \prod_{i=0}^{l-1} \frac{Z(n-i-1,\beta)}{Z(n-i,\beta)}\le \prod_{i=0}^{l-1}\exp\biggl(\frac{7\beta(n-i-1)}{3}\biggr)\\
&= \exp\biggl(\frac{7\beta n l}{3} -\frac{7\beta l(l+1)}{6}\biggr).
}
To complete the proof, note that  $l=-k$ and $l(l+1)= k(k-1)$.
\end{proof}
\subsection{Proofs of the upper bounds}\label{proof3d}
Let us now fix some $n\ge 0$ and $\beta>0$. In the following, $(X_1,\ldots, X_n)$ will denote a random configuration drawn from the measure $\mu_{n,\beta}$. We will assume that $(X_1,\ldots,X_n)$ is defined on some abstract probability space $(\Omega, \mf, \pp)$. Expectation, variance and covariance with respect to $\pp$ will be denoted by $\ee$, $\var$ and $\cov$ respectively.
\begin{lmm}\label{varlmm}
Let $D_1,\ldots,D_8$ denote the $8$ elements of $\dm_1$, and for each $1\le i\le 8$, let $N_i := |\{j: X_j\in D_i\}|$. Then for each $i$, $\ee(N_i)= n/8$ and 
\[
\var(N_i)\le K(\beta) n^{2/3},
\]
where $K(\beta)$ is a non-increasing function of $\beta$.
\end{lmm}
\begin{proof}
We have already defined universal constants $C_1$ and $C_2$ in the previous subsections. In this proof, we will continue to use this convention and denote further universal constants by $C_3, C_4,\ldots$ without explicitly mentioning that they denote universal constants on each occasion.

The identity $\ee(N_i)=n/8$ follows by symmetry. We will now prove the claimed bound on the variance. The cases $n=0$ and $n=1$ are trivial, so let us assume that $n\ge 2$. First, note that energy of a configuration is the sum of the energies within each $D_i$, plus the interactions between the $D_i$'s. From this observation it is easy to deduce the recursive relation
\eq{
Z(n,\beta) &= \sum_{\substack{0\le n_1,\ldots, n_8\le n\\ n_1+\cdots +n_8=n}} \frac{n!}{n_1!n_2!\cdots n_8!} e^{-2\beta\sum_{1\le i<j\le 8} n_i n_j}\prod_{i=1}^8 (8^{-n_i}Z(n_i,2\beta))\\
&= \sum_{\substack{0\le n_1,\ldots, n_8\le n\\ n_1+\cdots +n_8=n}} \frac{8^{-n}n!}{n_1!n_2!\cdots n_8!} e^{-2\beta\sum_{1\le i<j\le 8} n_i n_j}\prod_{i=1}^8 Z(n_i,2\beta). 
}
Moreover, for any $(n_1,\ldots, n_8)$ occurring in the above sum,
\eq{
&\pp(N_1=n_1,\ldots, N_8=n_8) \\
&= \frac{8^{-n}n!}{n_1!n_2!\cdots n_8!} e^{-2\beta\sum_{1\le i<j\le 8} n_i n_j}\frac{\prod_{i=1}^8 Z(n_i,2\beta) }{Z(n,\beta)}.
}
Choose nonnegative integers $m_1,\ldots, m_8$ such that $m_1+\cdots+m_8=n$ and $|m_i-n/8|\le 1$ for each $i$. It is not difficult to see that such integers can be found for any $n$. For convenience, let
\eq{
f(n_1,\ldots, n_8) &:= \frac{n!}{n_1!n_2!\cdots n_8!},\\
g(n_1,\ldots, n_8) &:= e^{-2\beta\sum_{1\le i<j\le 8} n_i n_j} = e^{-\beta n^2 + \beta \sum_{i=1}^8 n_i^2}, \\
h(n_1,\ldots, n_8) &:= \prod_{i=1}^8 Z(n_i,2\beta).
}
Take any $k_1,\ldots,k_8\in \zz$ such that $k_1+\cdots+k_8=0$ and $0\le m_i+k_i\le n$ for each~$i$. Then by Corollary \ref{zncor},
\eq{
&\frac{h(m_1+k_1,\ldots, m_8+k_8)}{h(m_1,\ldots, m_8)} \\
&\le \prod_{i=1}^8 \exp\biggl(-\frac{14\beta m_ik_i}{3} -\frac{14\beta k_i(k_i-1)}{6} + 2C_2 \beta |k_i| (n+|k_i|)^{1/3}\biggr)\\
&\le \prod_{i=1}^8 \exp\biggl(-\frac{14\beta (nk_i/8 - |k_i|)}{3} -\frac{14\beta k_i(k_i-1)}{6} + 4C_2\beta |k_i|n^{1/3}\biggr)\\
&\le \exp\biggl(-\frac{14\beta}{6}\sum_{i=1}^8 k_i^2 + C_3\beta n^{1/3}\sum_{i=1}^8 |k_i|\biggr).
}
Next, note that
\eq{
\frac{g(m_1+k_1,\ldots, m_8+k_8)}{g(m_1,\ldots, m_8)} &= \exp\biggl(\beta\sum_{i=1}^8 (m_i+k_i)^2 -\beta\sum_{i=1}^8 m_i^2\biggr)\\
&= \exp\biggl(\beta\sum_{i=1}^8(2m_ik_i + k_i^2)\biggr)\\
&\le \exp\biggl(\beta\sum_{i=1}^8(2nk_i/8 +2|k_i|+ k_i^2)\biggr)\\
&= \exp\biggl(\beta\sum_{i=1}^8(2|k_i|+ k_i^2)\biggr).
}
Therefore,
\eq{
&\frac{\pp(N_1=m_1+k_1,\ldots, N_8=m_8+k_8)}{\pp(N_1=m_1,\ldots, N_8=m_8)}\\
&\le \frac{f(m_1+k_1,\ldots, m_8+k_8)}{f(m_1,\ldots, m_8)} \exp\biggl(-\frac{4\beta}{3}\sum_{i=1}^8 k_i^2 + C_4 \beta n^{1/3} \sum_{i=1}^8|k_i|\biggr).}
This shows that there are positive constants $C_5$ and $C_6$ such that if 
\[
\max_{1\le i\le 8} |k_i|\ge C_5n^{1/3},
\]
then 
\eeq{
&\frac{\pp(N_1=m_1+k_1,\ldots, N_8=m_8+k_8)}{\pp(N_1=m_1,\ldots, N_8=m_8)}\\
&\le \frac{f(m_1+k_1,\ldots, m_8+k_8)}{f(m_1,\ldots, m_8)} e^{-C_6\beta n^{2/3}}. \label{fineq}
}
Let $A$ denote the set of all $(n_1,\ldots, n_8)$ such that each $n_i$ is a nonnegative integer, $n_1+\cdots+n_8=n$, and 
\[
\max_{1\le i\le 8} |n_i-m_i|\ge C_5 n^{1/3}. 
\]
Then by \eqref{fineq},  for any $(n_1,\ldots, n_8)\in A$,
\eq{
\frac{\pp(N_1=n_1,\ldots, N_8=n_8)}{\pp(N_1=m_1,\ldots, N_8=m_8)}&\le \frac{f(n_1,\ldots, n_8)}{f(m_1,\ldots, m_8)} e^{-C_6\beta n^{2/3}}.
}
Now recall the multinomial formula 
\eq{
\sum_{\substack{0\le n_1,\ldots,n_8\le n\\n_1+\cdots +n_8=n}} f(n_1,\ldots, n_8) = 8^n.
}
A simple calculation using Stirling's formula shows that
\eq{
f(m_1,\ldots, m_8)8^{-n} \ge C_{7} n^{-4}.
}
Thus,
\eq{
\pp((N_1,\ldots, N_8)\in A) &\le \frac{\pp((N_1,\ldots, N_8)\in A)}{\pp(N_1=m_1,\ldots, N_8=m_8)}\\
&= \sum_{(n_1,\ldots, n_8)\in A} \frac{\pp(N_1=n_1,\ldots, N_8=n_8)}{\pp(N_1=m_1,\ldots, N_8=m_8)}\\
&\le e^{-C_6\beta n^{2/3}} \frac{8^n}{f(m_1,\ldots, m_8)}\le C_{8} n^4e^{-C_6\beta n^{2/3}}. 
}
Therefore for each $i$,
\eq{
\var(N_i) &\le \ee(N_i-m_i)^2\le C_5^2 n^{2/3} + n^2 \pp((N_1,\ldots,N_8)\in A)\\
&\le C_5^2 n^{2/3}+ C_{8}n^6 e^{-C_6\beta n^{2/3}}. 
}
The above inequality shows that
\[
\var(N_i)\le K(\beta) n^{2/3},
\]
where $K(\beta)$ is a decreasing function of $\beta$. 
\end{proof}
For any Borel set $A\subseteq [0,1)^3$, let 
\[
X(A) := \{X_j: X_j\in A\}.
\]
and let $N(A):=|X(A)|$.  
For each $k\ge 0$, let $\mf_k$ be the $\sigma$-algebra generated the random variables $\{N(D): D\in\dm_k\}$. Note that $\{\mf_k\}_{k\ge 0}$ is a filtration of $\sigma$-algebras. This filtration will play an important role in the subsequent discussion.
\begin{lmm}\label{condlmm}
Conditional on $\mf_k$, the random sets $\{X(D): D\in \dm_k\}$ are mutually independent. Moreover, for any $D\in \dm_k$, conditional on $\mf_k$, $X(D)$ has the same distribution as a scaled version of a point process from the measure $\mu_{N(D), 2^k\beta}$.
\end{lmm}
\begin{proof}
Take any $k$. Note that the joint density of $(X_1,\ldots,X_n)$ at a point $(x_1,\ldots,x_n)$ may be written as 
\eq{
\frac{1}{Z(n,\beta)} \exp\biggl(-\beta\sum_{D\in \dm_k} H_D(x_1,\ldots,x_n) -\beta R_k(x_1,\ldots,x_n)\biggr),
}
where $H_D(x_1,\ldots,x_n)$ is the contribution due to the interactions between points in $D$, and $R_k(x_1,\ldots,x_n)$ is the contribution due to the interactions between points in different members of $\dm_k$. The crucial property of the potential $w$ is that $R_k(x_1,\ldots,x_n)$ is a function of $\{n_D: D\in \dm_k\}$, where $n_D = |\{j: x_j\in D\}|$. The claims  follow easily from this observation.
\end{proof}
Lemma \ref{condlmm} allows us to compute conditional means and variances.
\begin{lmm}\label{condcalc}
If $D\in \dm_k$ and $D'$ is a child of $D$, then  
\[
\ee(N(D')|\mf_k) = \frac{N(D)}{8}
\]
and
\eq{
\var(N(D')|\mf_k) &\le K(\beta) N(D)^{2/3},
}
where $K$ is the function from Lemma \ref{varlmm}.
\end{lmm}
\begin{proof}
The formula for the conditional expectation follows from Lemma~\ref{condlmm} and symmetry, and the bound on the conditional variance follows from Lemma \ref{condlmm}, Lemma \ref{varlmm}, and the observation that $K(2^k\beta)\le K(\beta)$ since $K$ is a non-increasing function of $\beta$.
\end{proof}
The above lemma leads to the following conclusions about unconditional means and variances.
\begin{lmm}\label{expvarlmm}
For any $D\in \dm$, $\ee(N(D))=\leb(D)n$ and 
\[
\var(N(D)) \le 8K(\beta) \leb(D)^{2/3}n^{2/3},
\]
where $K$ is the function from Lemma \ref{varlmm}.% and $\vol(D)$ is the volume  of~$D$.
\end{lmm}
\begin{proof}
Suppose that $D\in \dm_k$. The formula for the expectation follows easily by iterating the formula for the conditional expectation from Lemma \ref{condcalc}, and observing that $\leb(D)=8^{-k}$. Next, let $D'$ be the parent of $D$. Then by Lemma \ref{condcalc} and the formula for expected value,
\eq{
\ee(N(D)^2) &= \ee(N(D)^2 - (\ee(N(D)|\mf_{k-1}))^2) + \ee((\ee(N_{D}|\mf_{k-1}))^2)\\
&= \ee(\var(N(D)|\mf_{k-1})) + 8^{-2} \ee(N(D')^2)\\
&\le K(\beta) \ee(N(D')^{2/3}) + 8^{-2} \ee(N(D')^2)\\
&\le K(\beta) (\ee(N(D')))^{2/3} + 8^{-2} \ee(N(D')^2)\\
&= K(\beta) 4^{-k+1} n^{2/3} + 8^{-2}\ee(N(D')^2). 
}
Iterating this, we get
\eq{
\ee(N(D)^2) &\le K(\beta)n^{2/3}(4^{-k+1} + 8^{-2} 4^{-k+2} + 8^{-4} 4^{-k+3} + \cdots) + 8^{-2k}n^2\\
&\le 8K(\beta)4^{-k} n^{2/3} + 8^{-2k}n^2,
}
which completes the proof since $\ee(N(D)) = \leb(D)n = 8^{-k}n$.
\end{proof}
Now take any nonempty open set $U\subseteq [0,1)^3$ with regular boundary. Let $\uu$ be the set of all $D\in\dm$ such that $D\subseteq U$ but the parent cube of $D$ is not contained in $U$. 
\begin{lmm}\label{ulmm}
The set $U$ is the disjoint union of all elements of $\uu$.
\end{lmm}
\begin{proof}
Since $U$ is open, each point in $U$ belongs to some dyadic cube that is contained in $U$. Some ancestor of this cube must belong to $\uu$. This shows that $U$ is the union of the members of $\uu$. It is easy to see that the elements of $\uu$ are disjoint.
\end{proof}
\begin{cor}\label{expcor}
$\ee(N(U))=\leb(U)n$.
\end{cor}
\begin{proof}
Just observe that by Lemma \ref{ulmm} and Lemma \ref{expvarlmm},
\[
\ee(N(U)) = \sum_{D\in \uu} \ee(N(D)) = \sum_{D\in \uu} \leb(D)n = \leb(U)n, 
\]
where we have implicitly used the fact that $\uu$ is a countable collection.
\end{proof}
For each $j$, let $\uu_j := \uu\cap \dm_j$. Let $\vv_j$ denote the set of all $D\in \dm_j$ that intersect both $U$ and $U^c$. Note that $\uu_j$ and $\vv_j$ do not overlap. For any dyadic cube $D$, let $p(D)$ denote the proportion of $D$ that belongs to $U$. Let $M_0 = \leb(U)n$ and for each $j\ge 1$, let
\eq{
M_j := \sum_{i=0}^j\sum_{D\in \uu_i} N(D) + \sum_{D\in \vv_j} p(D) N(D). 
}
\begin{lmm}\label{martingale}
The sequence $\{M_j\}_{j\ge 0}$ is a martingale with respect to the filtration $\{\mf_j\}_{j\ge 0}$.
\end{lmm}
\begin{proof}
Take any $j\ge 1$. Then
\eq{
\ee(M_j|\mf_{j-1}) &= \sum_{i=0}^{j-1} \sum_{D\in \uu_i} N(D) + \sum_{D\in \uu_j} \ee(N(D)|\mf_{j-1}) \\
&\qquad + \sum_{D\in \vv_j} p(D) \ee(N(D)|\mf_{j-1})\\
&=  \sum_{i=0}^{j-1} \sum_{D\in \uu_i} N(D) + \sum_{D\in \uu_j\cup \vv_j} p(D) \ee(N(D)|\mf_{j-1})
}
Take any $D\in \vv_{j-1}$. Then each child of $D$ is either a member of $\uu_j$, or a member of $\vv_j$, or has no intersection with $U$. Conversely, every member of $\uu_j\cup \vv_j$ is the child of some member of $\vv_{j-1}$. Lastly, note that if $D_1,\ldots,D_8$ are the children of a dyadic cube $D$, then 
\[
p(D) = \frac{1}{8}\sum_{i=1}^8 p(D_i).
\]
Combining these observations and applying Lemma \ref{condcalc}, we get
\[
\sum_{D\in \uu_j\cup \vv_j} p(D) \ee(N(D)|\mf_{j-1}) = \sum_{D\in \vv_{j-1}} p(D) N(D),
\]
which completes the proof.
\end{proof}
For the remainder of this section, let $A(U)$ be a constant such that for all $0<\ep \le \diam(U)$,
\eeq{
\leb(\partial U_\ep)\le A(U)\ep. \label{audef}
}
By the regularity condition, we can choose $A(U)$ to be finite.
The martingale property  of $M_j$ and our previous calculations lead to the following conclusion.
\begin{lmm}\label{mjvar}
For any $j\ge 1$ such that $\sqrt{3}\cdot 2^{-j+1}\le \diam(U)$, 
\[
\var(M_j)\le C(\beta)A(U) n^{2/3} + \var(M_{j-1}),
\]
where $C(\beta)$ is a constant that depends only on $\beta$. %, and $\area(\partial U)$ is the area of the boundary of $U$. %Consequently, $\var(M_j)\le C(\beta)S(U) jn^{2/3}$.
\end{lmm}
\begin{proof}
By the martingale property,
\eeq{
\var(M_j) &= \ee(\var(M_j|\mf_{j-1})) + \var(\ee(M_j|\mf_{j-1}))\\
&= \ee(\var(M_j|\mf_{j-1})) + \var(M_{j-1}). \label{varvar}
}
Now,
\eeq{
\var(M_j|\mf_{j-1}) &= \var\biggl(\sum_{D\in \uu_j\cup \vv_j} p(D)N(D)\biggl|\mf_{j-1}\biggr)\\
&= \sum_{D,D'\in \uu_j\cup \vv_j} p(D)p(D') \cov(N(D), N(D')|\mf_{j-1}).\label{mjmj}%\\
%&\le \sum_{D,D'\in \uu_j\cup \vv_j}|\cov(N(D), N(D')|\mf_{j-1})|.
}
If $D$ and $D'$ have different parents, then $N(D)$ and $N(D')$ are conditionally independent by Lemma \ref{condlmm}, and hence the conditional covariance is zero. Otherwise, Lemma \ref{condcalc} and the Cauchy--Schwarz inequality imply that 
\[
|\cov(N(D), N(D')|\mf_{j-1})|\le K(\beta)N(D'')^{2/3},
\]
where $D''$ is the parent of $D$ and $D'$. 
Thus, by Lemma \ref{expvarlmm},
\eq{
|\ee(\cov(N(D), N(D')|\mf_{j-1})) | &\le K(\beta) (\leb(D'') n)^{2/3} \\
&= K(\beta) (8^{-j+1} n)^{2/3}.
}
On the other hand,  each $D\in \uu_j\cup \vv_j$ has at most $7$ sibling cubes that belong to $\uu_j\cup \vv_j$. Since $p(D)8^{-j} = p(D)\leb(D) = \leb(D\cap U)$, this shows that 
\eq{
\ee(\var(M_j|\mf_{j-1})) &\le K(\beta) (8^{-j+1}n)^{2/3}\sum_{D\in \uu_j\cup \vv_j} 7p(D)\\
&= 28K(\beta) n^{2/3}2^j\sum_{D\in \uu_j\cup \vv_j}\leb(D\cap U).
%&\le 28K(\beta) n^{2/3}\leb(U).
}
Note that each element of 
\[
\bigcup_{D\in \uu_j\cup \vv_j}( D\cap U)
\]
is within distance $\sqrt{3}\cdot 2^{-j+1}$ from $\partial U$. Since $\sqrt{3}\cdot 2^{-j+1}\le \diam(U)$, inequality~\eqref{audef} gives 
\[
\sum_{D\in \uu_j\cup \vv_j}\leb(D\cap U)\le  A(U) \sqrt{3}\cdot 2^{-j+1}.
\]
Consequently,
\[
\ee(\var(M_j|\mf_{j-1})) \le C(\beta) A( U) n^{2/3},
\]
where $C(\beta)$ depends only on $\beta$. 
The proof is completed by plugging this bound into \eqref{varvar}. %For the second assertion, observe that $M_0$ is a constant and hence $\var(M_0)=0$. 
\end{proof}
We now have all the ingredients for proving the following theorem, which implies Theorems \ref{macthm} and \ref{micthm} and special cases.
\begin{thm}[Hyperuniformity at all scales]\label{mainthm}
Let $U$ and $N(U)$ be as in Theorem~\ref{macthm}. Suppose that $\diam(U)\ge n^{-1/3}$. Let $A(U)$ be the constant defined in \eqref{audef}. Then 
\[
\ee(N(U)) = \leb(U)n
\]
and 
\eq{
\var(N(U))&\le C(\beta)A(U)n^{2/3}\log (4n^{1/3}\diam(U))+ C(\beta)\leb(U)^{2/3}n^{2/3},
}
where $C(\beta)$ is a constant that depends only on $\beta$. %, $\area(\partial U)$ is the area of the boundary of $U$, $\diam(U)$ is the diameter of $U$, and $\vol(U)$ is the volume of $U$.
\end{thm}
\begin{proof}
Throughout this proof, $C(\beta)$ will denote any constant that depends only on $\beta$. The value of $C(\beta)$ may change from line to line or even within a line. 

The formula for the expectation follows from Corollary \ref{expcor}. It remains to prove the variance bound. Choose $k$ such that
\[
\frac{1}{2}n^{-1/3}\le \sqrt{3} \cdot 2^{-k}\le n^{-1/3}. 
\]
Note that by Lemma \ref{ulmm}, any point in $U$ either belongs to some $D\in \uu_j$ for some $j\le k$, or belongs to some $D\in\uu_j$ for some $j> k$. In the latter case, there is an ancestor of $D$ that belongs to $\vv_{k}$. Thus, 
\eq{
U = \biggl(\bigcup_{j=0}^{k}\uu_j\biggr)\cup \biggl(\bigcup_{D\in \vv_{k}} (D\cap U)\biggr), 
}
and so
\eq{
N(U) = \sum_{j=0}^{k} \sum_{D\in \uu_j} N(D) + \sum_{D\in \vv_{k}} N(D\cap U).
}
Consequently, by Lemma \ref{condlmm}, Lemma \ref{expvarlmm} and Corollary \ref{expcor},
\eq{
\ee(N(U)|\mf_{k}) &= \sum_{j=0}^{k} \sum_{D\in \uu_j} N(D) + \sum_{D\in \vv_{k}} \ee(N(D\cap U)|\mf_{k})\\
&=  \sum_{j=0}^{k} \sum_{D\in \uu_j} N(D) + \sum_{D\in \vv_k} p(D)N(D) = M_k.
}
Therefore,
\eeq{
\var(N(U)) &= \ee(\var(N(U)|\mf_k)) + \var(\ee(N(U)|\mf_k))\\
&=  \ee(\var(N(U)|\mf_k)) + \var(M_k).\label{numk}
}
Given $\mf_k$, the random variables $\{N(D\cap U):D\in \dm_k\}$ are independent by Lemma~\ref{condlmm}. Therefore, by Lemma \ref{expvarlmm} and Corollary \ref{expcor},
\eq{
\var(N(U)|\mf_k) &= \var\biggl(\sum_{D\in \vv_k} N(D\cap U)\biggl|\mf_k\biggr)= \sum_{D\in \vv_k} \var(N(D\cap U)|\mf_k)\\
&\le \sum_{D\in \vv_k} \ee(N(D\cap U)^2|\mf_k)\\
&\le \sum_{D\in \vv_k} \ee(N(D\cap U)|\mf_k) N(D) = \sum_{D\in \vv_k} p(D) N(D)^2. 
}
By Lemma \ref{expvarlmm} and our choice of $k$, 
\[
\ee(N(D)^2)= \var(N(D))+(\ee(N(D)))^2 \le C(\beta) 
\]
for all $D\in \vv_k$. Also, each element of 
\[
\bigcup_{D\in \vv_k} (D\cap U) 
\]
is within distance $\sqrt{3}\cdot 2^{-k}$ of $\partial U$, and $p(D) 8^{-k} = \leb(D\cap U)$. Since 
\[
\sqrt{3}\cdot 2^{-k}\le n^{-1/3}\le \diam(U)
\]
 by our choice of $k$ and the assumption that $\diam(U)\ge n^{-1/3}$, this gives
\eq{
\ee(\var(N(U)|\mf_k)) &\le C(\beta)8^k\sum_{D\in \vv_k} \leb(D\cap U) \\
&\le C(\beta)8^kA(U) 2^{-k}\\
&= C(\beta)A(U)4^k \le C(\beta)A(U)n^{2/3}.
}
Let $l$ be the smallest integer such that $\sqrt{3}\cdot 2^{-l} \le \diam(U)$. Note that $l\le k$. 
Together with \eqref{numk} and Lemma \ref{mjvar}, the above inequality shows that
\eq{
\var(N(U))\le C(\beta)A(U)n^{2/3}(k-l+1) + \var(M_l).
}
By the definition of $l$, $\uu_i$ if empty for all $i<l$. Therefore
\[
M_l = \sum_{D\in \uu_l\cup \vv_l} p(D)N(D). 
\]
Note that for any $D\in \uu_l\cup \vv_l$, Lemma \ref{expvarlmm} gives
\eq{
\var(p(D)N(D)) &= p(D)^2\var(N(D))\le C(\beta)p(D)^2 \leb(D)^{2/3}n^{2/3}\\
&\le C(\beta)(p(D)\leb(D))^{2/3}n^{2/3}\\
&= C(\beta)\leb(D\cap U)^{2/3} n^{2/3}\le C(\beta)\leb(U)^{2/3}n^{2/3}. 
}
Moreover, it is easy to see that $U$ intersects at most $64$ members of $\dm_l$, and therefore $|\uu_l \cup \vv_l|\le 64$. From these observations, we get
\eq{
\var(M_l)\le C(\beta) \leb(U)^{2/3}n^{2/3}. 
}
Finally, note that by the lower bound on $\sqrt{3}\cdot 2^{-k}$ and the upper bound on $\sqrt{3}\cdot 2^{-l}$, we get 
\[
2^{k-l} \le 2n^{1/3}\diam(U),
\]
and hence $k-l+1\le \log_2(4n^{1/3}\diam(U))$. 
This completes the proof of the theorem. 
\end{proof}

\begin{proof}[Proof of Theorem \ref{macthm}] 
This is a direct application of Theorem \ref{mainthm}. The condition $\diam(U)\ge n^{-1/3}$ is irrelevant because the variance bound can be enforced for small $n$ by adjusting the constant $C(U,\beta)$.
\end{proof}

\begin{proof}[Proof of Theorem \ref{micthm}]
Let $V := n^{-1/3}\lambda U + x$. Note that $N_x(\lambda U)=N(V)$.  Also, note that 
\eq{
\leb(V) &= \lambda^3n^{-1}\leb(U),\\
A(V)&=\lambda^2n^{-2/3}A(U),\\
\diam(V) &= \lambda n^{-1/3} \diam(U).
}
In particular, the condition $\diam(V)\ge n^{-1/3}$ is equivalent to $\diam(\lambda U)\ge 1$. The proof is now just an  application of Theorem \ref{mainthm}, and the observation that since $x\in (0,1)^3$, $V$ is eventually contained in $(0,1)^3$ as $n$ gets large.
\end{proof}

Finally, let us prove Theorem \ref{linearthm}.

\begin{proof}[Proof of Theorem \ref{linearthm}]
Here $C(\beta)$ denotes any constant that depends only on $\beta$. Let $f(D)$ be the average value of $f$ in a dyadic square $D\in\dm$. For each $k$, let $f_k$ be the function that is identically equal to  $f(D)$ within each $D\in \dm_k$. Let
\eq{
W_k := X(f_k).
}
By Lemma \ref{condlmm} and Lemma \ref{condcalc}, it is easy to see that $\{W_k\}_{k\ge 0}$ is martingale with respect to the filtration $\{\mf_k\}_{k\ge 0}$. Moreover, for any $k$,  
\eeq{
\ee(X(f)|\mf_k) = X(f_k). \label{xfk}
}
Now choose $k$ such that
\[
n^{-1/3}\le 2^{-k}\le 2n^{-1/3}. 
\]
Then by \eqref{xfk} and the martingale property of $\{W_j\}_{j\ge 0}$,
\eeq{
\var(X(f)) &= \ee(\var(X(f)|\mf_k)) + \sum_{j=1}^{k}\ee(\var(X(f_j)|\mf_{j-1})). \label{varxf}
}
Take any $j$. For each $D\in \dm_{j-1}$, let $c(D)$ denote the set of $8$ children of $D$. By Lemma \ref{condlmm} and Lemma \ref{condcalc},
\eq{
&\var(X(f_j)|\mf_{j-1}) = \var\biggl(\sum_{D\in \dm_j} f(D)N(D) \biggl|\mf_{j-1}\biggr)\\
&= \sum_{D\in \dm_{j-1}} \var\biggl(\sum_{D'\in c(D)} f(D')N(D')\biggl|\mf_{j-1}\biggr)\\
&= \sum_{D\in \dm_{j-1}} \ee\biggl(\biggl(\sum_{D'\in c(D)} f(D')N(D') - f(D)N(D)\biggr)^2\biggl|\mf_{j-1}\biggr).
} 
Now notice that for any $D\in \dm_{j-1}$, 
\eq{
&\sum_{D'\in c(D)} f(D')N(D') - f(D)N(D) \\
&= \sum_{D'\in c(D)} (f(D')- f(D))\biggl(N(D')-\frac{N(D)}{8}\biggr).
}
Recall that $L$ is the Lipschitz constant of $f$. For any $D'\in c(D)$,
\[
|f(D')-f(D)|\le \sqrt{3} L 2^{-j+1}.
\]
Thus, 
\eq{
&\biggl(\sum_{D'\in c(D)} (f(D')- f(D))\biggl(N(D')-\frac{N(D)}{8}\biggr)\biggr)^2 \\
&\le 4^{-j+2} L^2 \biggl(\sum_{D'\in c(D)} \biggl|N(D')-\frac{N(D)}{8}\biggr|\biggr)^2\\
&\le 4^{-j+4}L^2\sum_{D'\in c(D)} \biggl(N(D')-\frac{N(D)}{8}\biggr)^2.
}
Therefore, by Lemma \ref{condcalc}, 
\eq{
&\ee\biggl(\biggl(\sum_{D'\in c(D)} f(D')N(D') - f(D)N(D)\biggr)^2\biggl|\mf_{j-1}\biggr)\\
&\le 4^{-j+4} L^2 \sum_{D'\in c(D)}\var(N(D')|\mf_{j-1})\le 4^{-j+6} L^2 K(\beta)N(D)^{2/3}.
}
Consequently, by Lemma \ref{expvarlmm},
\eq{
&\ee(\var(X(f_j)|\mf_{j-1})) \le C(\beta)L^2 4^{-j}\sum_{D\in \dm_{j-1}}\ee(N(D)^{2/3})\\
&\le C(\beta) L^24^{-j}\sum_{D\in \dm_{j-1}}(\ee(N(D)))^{2/3}\le C(\beta) L^22^{-j}n^{2/3}.
}
Next, for $D\in \dm_k$, let
\[
s(D) := \sum_{j\, : \, X_j\in D} f(X_j),
\]
so that
\[
X(f) = \sum_{D\in \dm_k} s(D). 
\]
Then by Lemma \ref{condlmm},
\eq{
\var(X(f)|\mf_k) &= \sum_{D\in \dm_k} \var(s(D)|\mf_k)\\
&\le \sum_{D\in \dm_k}\ee((s(D)-f(D)N(D))^2|\mf_k). 
}
By the Lipschitz condition,
\eq{
|s(D)-f(D)N(D)|\le \sqrt{3}L 2^{-k}N(D)
}
for each $D\in \dm_k$. Thus, by Lemma \ref{expvarlmm} and our choice of $k$,
\eq{
\ee((s(D)-f(D)N(D))^2) &\le 4^{-k+1}L^2 \ee(N(D)^2)\le C(\beta)L^2 4^{-k}.
}
Consequently,
\eq{
\ee(\var(X(f)|\mf_k)) &\le C(\beta)L^24^{-k}|\dm_k|\le C(\beta) L^22^k\le C(\beta)L^2 n^{1/3}.
}
The proof is now easily completed by combining the steps. 
\end{proof}

\subsection{Proofs of the lower bounds}
Let us now prove Theorem \ref{lowthm}. We will continue using the notations introduced in the previous sections. We  need to prove some simple geometric facts. Let 
\[
\mt:= \{z+[0,1)^3: z\in \zz^3\}. 
\] 
Our first geometric lemma is very simple.
\begin{lmm}\label{geomlmm0}
Let $\mt$ be as above. Take any $D\in \mt$ and any $x\in D$. Let $\delta$ be the distance of $x$ from the boundary of $D$. Then any plane through $x$ bifurcates $D$ into two parts, each of which has volume at least $2\pi \delta^3/3$. 
\end{lmm}
\begin{proof}
The open ball of radius $\delta$ around $x$ is contained in $D$. Any plane $P$ through $x$ bifurcates this ball into two parts of volume $2\pi \delta^3/3$ each. The proof is completed by observing that these  two hemispheres are contained in the two parts of $D$ obtained by bifurcating using $P$. 
\end{proof}
The second lemma is an easy fact about intervals.
\begin{lmm}\label{intlmm}
Let $I$ be a closed interval of the real line of length at least $\delta\in [0,1]$. Then $I$ has a closed subinterval $J$ of length $\delta/4$ such that any integer is at a distance at least $\delta/4$ from $J$.
\end{lmm}
\begin{proof}
If $I$ contains no integers, then we can take $J$ to be an interval of length $\delta/4$ that is at distance at least $\delta/4$ from each endpoint of $I$. If $I$ contains an integer $n$, then at least one of the two intervals $[n, n+\delta/2]$ and $[n-\delta/2, n]$ must be contained in $I$. In the first case  take $J= [n+\delta/4, n+\delta/2]$ and in the second case take $J=[n-\delta/2, n-\delta/4]$. Since $\delta\le 1$, there is no integer within distance $\delta/4$ from $J$. 
\end{proof}
The next lemma is intuitively obvious but a little tedious to prove. The constants are probably not optimal, but that does not matter for us.
\begin{lmm}\label{geomlmm}
Take any $x\in \rr^3$ and a unit vector $u = (u_1,u_2,u_3)\in S^2$. Let $P$ be the plane that contains $x$ and is perpendicular to $u$. Suppose that
\eeq{
\min\{|u_1|, |u_2|, |u_3|\} \ge 0.1.\label{u1u2u3}
}
Then there is an element $D\in \mathcal{T}$, within Euclidean distance $\sqrt{402}$ from $x$, which  is bifurcated by the plane $P$ in such a way that each part has volume at least $6\times 10^{-8}$.
\end{lmm}
\begin{proof}
Take any $x = (x_1,x_2,x_3)\in \rr^3$ and $u=(u_1,u_2, u_3)\in S^2$ as in the statement of the lemma.  Let $P_0$ be the plane with normal vector $u$ that contains the origin. Define 
\eq{
y_1= \textup{sign}(u_1),\ \ y_2 = \textup{sign}(u_2), \ \ y_3 = -\frac{|u_1|+|u_2|}{u_3}. 
}
Then $y= (y_1,y_2,y_3)\in P_0$. Also, we have $|y_1|=1$, $|y_2|=1$, and by condition~\eqref{u1u2u3} and the fact that $|u_3|\le 1$,
\[
|y_3| = \frac{|u_1|+|u_2|}{|u_3|} \ge |u_1|+|u_2|\ge0.2.
\]
Now consider the set
\[
I_1 = \{x_1 + \alpha y_1: 0\le \alpha\le 1\}.
\]
Since $|y_1|=1$, $I_1$ is an interval of length $1$. By Lemma \ref{intlmm}, $I_1$ has a subinterval of $I_2$ of length $0.25$ such that any integer is at least at a distance $0.25$ from $I_2$. Moreover, since $|y_1|= 1$, $I_2$ is of the form
\[
\{x_1+\alpha y_1: a\le \alpha\le b\},
\]
where $b-a= 0.25$. Let
\[
I_3 := \{x_2+\alpha y_2: a\le \alpha\le b\}. 
\]
Since $|y_2|= 1$, $I_3$ has length $0.25$. Thus by Lemma \ref{intlmm}, $I_3$ contains a subinterval $I_4$ of length $0.0625$ such that any integer is at a distance at least $0.0625$ from $I_4$. Again, since $|y_2|= 1$, this implies that $I_4$ is of the form
\[
\{x_2+\alpha y_2: c\le \alpha\le d\},
\] 
where $a\le c\le d\le b$ and $d-c = 0.0625$. Let 
\[
I_5 := \{x_3+\alpha y_3: c\le \alpha\le d\}.
\]
Since $|y_3|\ge 0.2$, $I_5$ has length at least $0.0125$. Consequently by Lemma~\ref{intlmm}, $I_5$ has a subinterval $I_6$ of length $0.003125$ such that any integer is at a distance at least $0.003125$ from $I_6$.

In particular, there is some $\alpha\in [0,1]$ such that $x_1+\alpha y_1\in I_2$, $x_2+\alpha y_2\in I_4$ and $x_3+\alpha y_3\in I_6$. The distance of $x_i+\alpha y_i$ from the nearest integer is at least $0.003125$ for each $i$. Thus, the distance of the point $x+\alpha y$ from the boundary of the cube $D\in \mt$ that contains $x+\alpha y$ is at least $0.003125$. By Lemma \ref{geomlmm0} and the fact that $x+\alpha y\in P$, this proves $P$ bifurcates $D$ into two parts, each of which has volume at least $6\times 10^{-8}$. Lastly, note that 
\eq{
|(x+\alpha y)-x|&\le |y| = \sqrt{y_1^2+y_2^2+y_3^2}\\
&\le \sqrt{1+1 + \frac{(1+1)^2}{0.1^2}}\le \sqrt{402},
}
since $|u_1|\le 1$, $|u_2|\le 1$ and $|u_3|\ge 0.1$. 
This completes the proof of the lemma. 
\end{proof}
Now recall that the boundary of the set $U$ in the statement of Theorem~\ref{lowthm} is a smooth, closed, orientable surface. In particular, we can choose a unit normal vector $u(x)$ at each $x\in \partial U$ such that the map $x\mapsto u(x)$ is smooth.
\begin{lmm}\label{dcapu}
Take any $x\in \partial U$ such that the normal vector $u(x)$ satisfies~\eqref{u1u2u3}. Then there is some $j_0$ depending only on $U$ (but not on $x$), such that for all $j\ge j_0$, there is some $D\in \dm_j$ at distance at most $\sqrt{402} \cdot 2^{-j}$ from $x$, which satisfies
\eeq{
10^{-8}\le \frac{\leb(D\cap U)}{\leb(D)}\le 1- 10^{-8}. \label{djbd}
}
\end{lmm}
\begin{proof}
From the given properties of $\partial U$, it is clear that $\partial U$ has uniformly bounded curvature. Consequently, there is a constant $C$ depending only on $U$, such that for any $x\in \partial U$  and any $\ep\in (0,1)$, $B(x,\ep)\cap \partial U$ lies inside a slab of width $C\ep^2$ around $T_x$, where $B(x,\ep)$ is the Euclidean ball of radius $\ep$ around $x$, and $T_x$ is the tangent plane at $x$. The rest of the proof is an easy application of Lemma~\ref{geomlmm} and scaling. 
\end{proof}
The above lemma leads to the following result, which is a key component of the proof of Theorem \ref{lowthm}.
\begin{lmm}\label{k1lmm}
There is some $K_1>0$ and some $j_1\ge 1$ depending only on $U$ such that for any $j\ge j_1$, there is a set of at least $K_1 4^j$ cubes $D\in \dm_j$ that satisfy~\eqref{djbd} and the union of these cubes has diameter at most $\diam(U)/3$.
\end{lmm}
\begin{proof}
Let $P$ be the plane through the origin that is perpendicular to the vector $(1,1,1)$. Let $\alpha_0$ be the largest $\alpha$ such that the plane $P_\alpha := (\alpha,\alpha,\alpha)+P$ intersects the closure of $U$. Let $x$ be a point of intersection. Then $x\in \partial U$, and $P_{\alpha_0}=T_x$. Consequently, there is some $0<\ep<\diam(U)/7$  such that for every $y\in B(x,\ep)\cap \partial U$, $u(y)$ satisfies \eqref{u1u2u3}. Due to the boundedness of the curvature of $\partial U$, a small enough choice of $\ep$ guarantees that for any $\delta\in (0,1)$, there are at least $C\delta^{-2}$ points in $B(x,\ep)\cap \partial U$, where $C$ is a positive constant that depends only on $U$, such that any two points are at distance at least $50\delta$ from each other. 

Take  $\delta=2^{-j}$, and choose a collection of points as above. Then by Lemma~\ref{dcapu}, there is an element of $\dm_j$ within distance $21\delta$ from each point, that satisfies \eqref{djbd}. Since the points are separated by distance at least $50\delta$ from each other, these elements of $\dm_j$ are distinct. Since $\ep<\diam(U)/7$, a large enough choice of $j$ ensures that the union of these cubes has diameter less than $\diam(U)/3$. 
\end{proof}
Lastly, we need  a lemma about our point process. Recall that for any $D\in \dm$, $N(D)$ is the number of points landing in $D$.
\begin{lmm}\label{repulsion}
For any $n\ge 1$, $\beta>0$, $j\ge 0$ and $D\in \dm_j$,
\[
\pp(N(D)\ge 2)\le \exp\biggl(-2^{j+1}\beta +\frac{7\beta}{3}{n\choose 2}\biggr).
\]
\end{lmm}
\begin{proof}
The $n=1$ case is trivial, so let us take $n\ge 2$. By Jensen's inequality and Lemma \ref{wlmm},
\[
Z(n,\beta) \ge \exp\biggl(-\frac{7\beta }{3}{n\choose 2}\biggr). 
\]
On the other hand, if a configuration $x_1,\ldots, x_n$ has two or more points in $D$, then 
\[
H_n(x_1,\ldots,x_n)\ge 2^{j+1}.
\]
Thus, if $A$ is the set of all such configurations, then
\[
\int_A e^{-\beta H_n(x_1,\ldots,x_n)}\, dx_1\cdots dx_n \le e^{- 2^{j+1}\beta }\leb(A)\le e^{-2^{j+1}\beta}. 
\]
Combining, we get
\eq{
\pp(N(D)\ge 2) = \mu_{n,\beta}(A) \le \exp\biggl(-2^{j+1}\beta +\frac{7\beta}{3}{n\choose 2}\biggr),
}
which completes the proof.
\end{proof}
Finally, we are ready to prove Theorem \ref{lowthm}. Recall the filtration $\{\mf_k\}_{k\ge 0}$ defined earlier.
\begin{proof}[Proof of Theorem \ref{lowthm}]
In this proof, the phrase `$n$ sufficiently large' will mean `$n\ge n_0$, where $n_0$ depends only on $U$ and $\beta$'. Also, $C$ will denote any positive universal constant, $C(\beta)$ will denote any positive constant that depends only on $\beta$, and $C(U,\beta)$ will denote any positive constant that depends only on $U$ and~$\beta$.

Choose $k$ such that 
\eeq{
n^{-1/3}\le 2^{-k} \le 2n^{-1/3}. \label{kdefeq}
}
Then for any $D\in \dm_k$, Lemma \ref{expvarlmm} gives
\eeq{
\ee(N(D)^2) \le K_2(\beta),\label{nd2}
}
where $K_2(\beta)$ is a positive integer that depends only on $\beta$. Let 
\[
m := 1000 K_2(\beta).
\]
Let $j>k$ be the smallest number such that 
\[
2^{j-k+1} \ge \frac{7}{3}{m\choose 2}+1. 
\]
Note that $0\le j-k\le C(\beta)$.

Take any $D\in \dm_k$. Let $\dm_j(D)$ denote the set of elements of $\dm_j$ that are descendants of $D$. Take any $D'\in \dm_j(D)$. If $N(D)\le m$, then by Lemma~\ref{repulsion} and Lemma~\ref{condlmm},
\[
\pp(N(D')\ge 2|\mf_k) \le e^{-2^k\beta} \le e^{-\beta n^{1/3}}. 
\]
Consequently,
\eq{
\pp(N(D)\le m, \, N(D')\ge 2) &= \ee(\pp(N(D')\ge 2|\mf_k); N(D)\le m)\\
&\le e^{-\beta n^{1/3}} \pp(N(D)\le m)\le e^{-\beta n^{1/3}}. 
}
In particular, if $E$ is the event
\eq{
&\{N(D)\le m \text{ and } N(D')\ge 2 \text{ for some } D\in \dm_k  \text{ and some } D'\in \dm_j(D)\},
}
then a union bound gives
\eeq{
\pp(E)&\le\sum_{D\in \dm_k}\sum_{D'\in \dm_j(D)}\pp(N(D)\le m, \, N(D')\ge 2)\\
&\le |\dm_j| e^{-\beta n^{1/3}} \le  C(\beta) ne^{-\beta n^{1/3}}.\label{ndnd}
}
We will need this inequality later.

Now, if $n$ is sufficiently large, then there is a set $\mc'\subseteq \dm_j$ that satisfies the conclusions of Lemma \ref{k1lmm}. In particular, $|\mc'|\ge C(U,\beta) 4^j$. Moreover, since each element of $\mc'$ satisfies \eqref{djbd}, these cubes must lie entirely within distance $\sqrt{3}\cdot 2^{-j}$ from $\partial U$. If $n$ is large enough, then $\sqrt{3}\cdot 2^{-j}\le \diam(U)$. Therefore by the regularity of $\partial U$, we have $|\mc'|\le C(U,\beta) 4^j$.

Let $\mc$ denote the set of all members of $\dm_k$ who are ancestors of elements of $\mc'$. By dropping some elements from $\mc'$ if necessary, we can ensure that each member of $\mc$ has exactly one descendant in $\mc'$. Since $0\le j-k\le C(\beta)$, this gives the inequalities 
\eeq{
C_1(U,\beta) 4^k\le |\mc|=|\mc'|\le C_2(U,\beta)4^k,\label{mcsize}
}
where $C_1(U,\beta)$ and $C_2(U,\beta)$ are positive constants that depend only on $U$ and $\beta$. 
Let $Q$ be the union of the elements of $\mc$. Recall that by Lemma \ref{k1lmm} and the relation between $\mc$ and $\mc'$, 
\[
\diam(Q)\le \frac{\diam(U)}{3}+2\sqrt{3}\cdot 2^{-k},
\] 
which is less than $\diam(U)/2$ if $n$ is sufficiently large. 
Thus, if $n$ is large enough and  $\sqrt{3}\cdot 2^{-k} \le \ep\le \diam(Q)$, then 
\[
\ep+\sqrt{3}\cdot 2^{-k}\le 2\ep\le 2\,\diam(Q)\le \diam(U).
\]
Moreover, each point in $Q$ is at distance at most $\sqrt{3}\cdot 2^{-k}$ from $U$. Therefore, 
\eq{
\leb(\partial Q_\ep) &\le \leb(\partial U_{\ep+\sqrt{3}\cdot 2^{-k}})\le A(U)(\ep+ \sqrt{3}\cdot 2^{-k} ) \le 2A(U)\ep.
}
On the other hand, if $0<\ep\le \sqrt{3}\cdot 2^{-k}$, then 
\eq{
\leb(\partial Q_\ep) &\le \sum_{D\in \mc} \leb(\partial D_\ep)\le \sum_{D\in \mc} A(D)\ep\le C\sum_{D\in \mc} 4^{-k}\ep= C |\mc|4^{-k}\ep.
}
%Now, the elements of $\mc$ are disjoint, have Lebesgue measure $8^{-k}$, and are contained in $\partial U_{\sqrt{3}\cdot 2^{-k}}$. Thus,
%\[
%|\mc|\le \frac{\leb(\partial U_{\sqrt{3}\cdot 2^{-k}})}{8^{-k}}\le \sqrt{3}A(U) 4^k.
%\]
Therefore, by \eqref{mcsize}, for $0<\ep\le \sqrt{3}\cdot 2^{-k}$,
\[
\leb(\partial Q_\ep)\le C(U,\beta)\ep.
\]
Combining the two cases, we get $A(Q)\le C(U,\beta)$.  Consequently, by Theorem \ref{mainthm}, 
\eeq{
\var(N(Q)) &\le C(U,\beta) n^{2/3}\log n, \label{varnq}
}
provided that $n$ is sufficiently large. 
Also, by Lemma \ref{expvarlmm} and our choice of~$k$,
\eq{
\ee(N(Q)) &= \leb(Q) n = |\mc| 8^{-k} n\ge |\mc|.
}
Thus, by \eqref{mcsize}, \eqref{varnq} and Chebychev's inequality,
\eeq{
\pp\biggl(\frac{N(Q)}{|\mc|}\ge \frac{1}{2}\biggr) &\ge 1-\frac{4\var(N(Q))}{|\mc|^2}\ge1-C(U,\beta) n^{-2/3}\log n.\label{a1ineq}
}
Now let 
\eq{
&a_1 :=\frac{1}{|\mc|}\sum_{D\in \mc} N(D) = \frac{N(Q)}{|\mc|}, \ \ \ a_2 := \frac{1}{|\mc|} \sum_{D\in \mc}N(D)^2,\\
&p_1 := \frac{|\{D\in \mc: N(D)>0\}|}{|\mc|},\ \  \ p_2 := \frac{|\{D\in \mc: N(D)> m\}|}{|\mc|},\\
&q := \frac{|\{D\in \mc: 0<N(D)\le m\}|}{|\mc|}.
}
By \eqref{nd2}, $\ee(a_2)\le K_2(\beta)$. Thus, 
\eeq{
\pp(a_2\ge 2K_2(\beta))\le \frac{1}{2}. \label{a2ineq}
}
By the Paley--Zygmund second moment inequality,
\eq{
p_1\ge \frac{a_1^2}{a_2}, 
}
and so by \eqref{a1ineq} and \eqref{a2ineq},
\eq{
\pp\biggl(p_1 \ge \frac{1}{8K_2(\beta)}\biggr) &\ge \pp\biggl(a_1\ge \frac{1}{2}, \, a_2 \le 2K_2(\beta)\biggr)\ge \frac{1}{2}-C(U,\beta) n^{-2/3}\log n. 
}
Choose $n$ so large that the above lower bound at least $1/3$. Next, note that by Lemma \ref{expvarlmm} and Markov's inequality,
\eq{ 
\ee(p_2) &\le \frac{1}{m|\mc|} \sum_{D\in \mc}\ee(N(D))\le \frac{8}{m},
}
and hence
\eq{
\pp\biggl(p_2\ge \frac{32}{m}\biggr)\le \frac{1}{4}.
}
Since $q=p_1-p_2$ and 
\[
\frac{1}{8K_2(\beta)} \ge \frac{64}{m},
\]
this gives  
\eeq{
\pp\biggl(q\ge \frac{32}{m}\biggr) &\ge \pp\biggl(p_1\ge \frac{64}{m},\, p_2\le \frac{32}{m}\biggr)\ge \frac{1}{3}-\frac{1}{4}= \frac{1}{12}.\label{qeq}
}
Let $\mc_0$ be the set of all $D\in \mc$ such that $0<N(D)\le m$. Let $\mc'_0$ be the set of all elements of $\mc'$ that are contained in elements of $\mc_0$. Let
\eeq{
r &:= \frac{1}{|\mc'_0|} \sum_{D\in \mc'_0}N(D) \label{rdefeq}
}
if $\mc'_0\ne \emptyset$ and let $r=0$ otherwise. 
By Lemma \ref{condcalc}, if $\mc'_0$ is nonempty, 
\eeq{
\ee(r|\mf_k) &= \frac{1}{8^{j-k}|\mc_0|} \sum_{D\in \mc_0}N(D) \ge C(\beta),\label{rexpeq}
}
and by Lemma~\ref{condcalc} and Lemma \ref{condlmm},
\eeq{
\var(r|\mf_k) &\le \frac{C(\beta)}{|\mc'_0|} = \frac{C(\beta)}{|\mc|q}\le \frac{C(\beta)}{n^{2/3}q}. \label{rvareq}
}
By the last two inequalities and Chebychev's inequality, we see that there is a positive constant $K_3(\beta)$ depending only on $\beta$ such that if $q\ge 32/m$ and $n$ is sufficiently large, then 
\[
\pp(r\ge K_3(\beta)|\mf_k) \ge 1-C(U,\beta)n^{-2/3}.
\]
Therefore by \eqref{qeq}, if $n$ is sufficiently large,
\eeq{
\pp\biggl(r\ge K_3(\beta), \, q\ge \frac{32}{m}\biggr)\ge \frac{1}{13}.\label{rqeq}
}
Thus, for sufficiently large $n$,
\eq{
\pp(|\mc_0'| \ge K_4(\beta) n^{2/3}) \ge \frac{1}{13},
}
where $K_4(\beta)$ is a positive constant that depends only on $\beta$.

Now recall the event $E$ defined earlier. Let $E^c$ denote the complement of~$E$. If $E^c$ happens, then $\mc_0'= \mc^*$, where 
\eeq{
\mc^* := \{D\in \mc_0': N(D)=1\}.\label{mcstar}
}
Combining this with \eqref{ndnd}, this shows that for sufficiently large $n$,
\eeq{
\pp(|\mc^*| \ge K_4(\beta)n^{2/3})&\ge \pp(\{|\mc_0'| \ge K_4(\beta) n^{2/3}\}\cap E^c)\\
&\ge\pp(|\mc_0'| \ge K_4(\beta) n^{2/3})-\pp(E)\ge \frac{1}{14}.\label{cstar}
}
By Lemma \ref{condlmm}, the random variables $\{N(D\cap U): D\in \dm_j\}$ are independent given $\mf_j$. If $N(D)=1$, then the conditional distribution of $N(D\cap U)$ given $\mf_j$ is Bernoulli$(p(D))$, where $p(D)=\leb(D\cap U)/\leb(D)$. Let
\[
M := \sum_{D\in \mc^*} N(D\cap U).
\]
Since $10^{-8} \le p(D)\le 1-10^{-8}$ for each $D\in \mc^*$, the Berry--Esseen theorem for sums of independent random variables shows that for any interval $I$,
\eeq{
\pp(M\in I|\mf_j)\le \frac{C(|I|+1)}{\sqrt{|\mc^*|}}, \label{berry}
}
where $|I|$ denotes the length of $I$. Since 
\[
N(U) = \sum_{D\in \dm_j} N(D\cap U) = \sum_{D\in \dm_j\setminus \mc^*}N(D\cap U) + M,
\]
and the two terms in the last expression are independent given $\mf_j$, the inequality \eqref{berry} implies that
\eq{ 
\pp(N(U)\in I|\mf_j)\le \frac{C(|I|+1)}{\sqrt{|\mc^*|}}.
}
Therefore by \eqref{cstar},
\eq{
\pp(N(U)\in I)\le C(\beta)(|I|+1)n^{-1/3} + \frac{13}{14}
}
if $n$ is sufficiently large. This completes the proof.
\end{proof}

Finally, let us prove Theorem \ref{linlowthm}. The ingredients are almost all drawn from the proof of Theorem \ref{lowthm}.

\begin{proof}[Proof of Theorem \ref{linlowthm}]
In this proof, $C(\beta)$ denotes any positive constant that depends only on $\beta$, $C(f)$ denotes any positive constant that depends only on $f$ and $C(f,\beta)$ denotes any positive constant that depends only on $f$ and $\beta$. Let $j$ and $k$ be defined as in \eqref{kdefeq}. Let $f:[0,1]^3\to\rr$ be a non-constant linear function. 

Let $\mc := \dm_k$, and let $a_1$, $a_2$, $p_1$, $p_2$ and $q$ be defined as in the proof of Theorem \ref{lowthm}, with this $\mc$. Then $|\mc|=8^k$, and $a_1 = 8^{-k}n\ge 1$. The inequality \eqref{a2ineq} is still valid, and hence we get
\eq{
\pp\biggl(p_1\ge \frac{1}{8K_2(\beta)} \biggr) \ge \frac{1}{2}. 
}
Proceeding then as in the proof of Theorem \ref{lowthm}, this gives
\eq{
\pp\biggl(q \ge \frac{32}{m}\biggr)\ge \frac{1}{4}. 
}
Let $\mc_0$ be the set of all $D\in \mc$ for which $0<N(D)\le m$. Construct a set $\mc_0'\subseteq \dm_j$ by choosing exactly one descendant of each element of $\mc_0$ by some arbitrary deterministic rule. Let $r$ be defined as in \eqref{rdefeq}. Then~\eqref{ndnd}, \eqref{rexpeq} and~\eqref{rvareq} continue to hold, and therefore so does~\eqref{rqeq} when $n$ is sufficiently large. Since $|\mc|\ge n$ in this proof, this shows that for sufficiently large $n$,
\eeq{
\pp(|\mc^*|\ge K_5(\beta)n) \ge \frac{1}{14},\label{cstar2}
} 
where $\mc^*$ is defined as in \eqref{mcstar} and $K_5(\beta)$ is a positive constant that depends only on $\beta$.

For each $D\in \dm_j$, let
\[
X(f,D) := \sum_{i\, : \, X_i\in D} f(X_i). 
\]
By Lemma \ref{condlmm}, the random variables $\{X(f,D): D\in \dm_j\}$ are conditionally independent given $\mf_j$.  Let
\[
M := n^{1/3}\sum_{D\in \mc^*} X(f, D).
\]
Now take any $D\in \mc^*$. Recall that $D$ contains exactly one point of our point process, and by Lemma \ref{condlmm}, the conditional distribution of this point given $\mf_j$ is uniform over the cube $D$. Since $f$ is a linear function, it is easy to see from this observation that for any $D\in \mc^*$, the conditional distribution of the random variable 
\[
n^{1/3}(X(f,D) - \ee(X(f,D)))
\]
given $\mf_j$ is actually non-random, and depends only on $f$. In particular, since $f$ is also non-constant, this shows that
\eq{
\var(n^{1/3}X(f,D)|\mf_j) = K_6(f)
}
and
\eq{
\ee\bigl(|n^{1/3}X(f,D) - \ee(n^{1/3}X(f,D))|^3\bigl|\mf_j\bigr)= K_7(f),
}
where $K_6(f)$ and $K_7(f)$ are strictly positive constants that depend only on~$f$.  Therefore by the Berry--Esseen theorem, for any interval $I$,
\eeq{
\pp(M\in I|\mf_j)\le \frac{C(f)(|I|+1)}{\sqrt{|\mc^*|}}, \label{berry2}
}
where $|I|$ denotes the length of $I$. Since 
\[
n^{1/3}X(f) = n^{1/3}\sum_{D\in \dm_j} X(f,D) = n^{1/3}\sum_{D\in \dm_j\setminus \mc^*}X(f,D) + M,
\]
and the two terms in the last expression are independent given $\mf_j$, the inequality \eqref{berry2} implies that
\eq{ 
\pp(n^{1/3}X(f)\in I|\mf_j)\le \frac{C(f)(|I|+1)}{\sqrt{|\mc^*|}}.
}
Therefore by \eqref{cstar2},
\eq{
\pp(n^{1/3}X(f)\in I)\le \frac{C(f,\beta)(|I|+1)}{\sqrt{n}} + \frac{13}{14}
}
if $n$ is sufficiently large. This completes the proof.
\end{proof}

\section{Proofs in 2D and 1D}
In this section, we will prove the results of Section \ref{results2d}. The proofs are similar to the proofs in the 3D case, but there are substantial differences, which is why we need a separate section. 
\subsection{Notation}
All notation will remain the same as in the 3D case. For example, $\dm_k$ will denote dyadic sub-squares of side-length $2^{-k}$ in 2D, and dyadic sub-intervals of length $2^{-k}$ in 1D. The main change is that $w$ is now different, namely, $w(x,y)=k(x,y)$, where $k(x,y)$ is the smallest $k$ such that $x$ and $y$ belong to distinct elements of $\dm_k$. The partition function $Z(n,\beta)$ and the measure $\mu_{n,\beta}$ are defined as before, with this new $w$ instead of the old one. We will denote the dimension by $d$, which may be $1$ or $2$.  %For convenience, we will continue to refer to members of $\dm$ as cubes instead of squares or intervals.

\subsection{Preliminary calculations}
First, let us carry out the calculations analogous to those done in Section \ref{prelimsec}.
\begin{lmm}\label{wlmm2d}
For each $x\in [0,1)^d$,
\[
\int w(x,y) \, dy = \frac{2^d}{2^d-1}.
\]
Consequently,
\[
\iint w(x,y) \, dx\, dy = \frac{2^d}{2^d-1}.
\]
\end{lmm}
\begin{proof}
Take any $x$. For each $k$, let $D_k$ be the element of $\dm_k$ that contains $x$. It is easy to see that the set of all $y$ with $w(x,y)=k$ is exactly the union of all members of $\dm_k$ that are contained in $D_{k-1}$, except the one that contains~$x$. The Lebesgue measure of this set is $2^{-dk} (2^d-1)$. Thus,
\[
\int w(x,y)\, dy = (2^d-1)\sum_{k=1}^\infty k2^{-dk} = \frac{2^d}{2^d-1}.
\]
The second assertion is obvious from the first.
\end{proof}
Let us now investigate energy-minimizing configurations of finite size. 
As before, $L_n$ will denote the minimum possible energy of a configuration of $n$ points. The following result gives upper and lower bounds for $L_n$ in dimensions one and two. 
\begin{thm}\label{lnlowthm2d}
There is a positive constant $C_1$ such that for each $n\ge 2$,
\[
{n\choose 2}\frac{2^d}{2^d-1} - C_1n\log n\le L_n \le {n\choose 2}\frac{2^d}{2^d-1}.
\]
\end{thm}
\begin{proof}
The proof of the upper bound is exactly the same as in Theorem \ref{lnlowthm}.  For the lower bound, let $k$ be an integer such that
\[
n^{-1/d}\le 2^{-k}\le 2n^{-1/d}. 
\]
Take any configuration of $n$ points. For each $D\in \dm$, let $n_D$ be the number of points in $D$. Summing up the contributions to the energy from each cube, we get
\eq{
H_n(x_1,\ldots, x_n) &= \sum_{j=1}^\infty \sum_{D\in \dm_j} {n_D\choose 2} + {n\choose 2}\\
&\ge \sum_{j=1}^k \sum_{D\in \dm_j} {n_D\choose 2} + {n\choose 2}= \frac{1}{2}\sum_{j=1}^k \sum_{D\in \dm_j} n_D^2 -  \frac{nk}{2} + {n\choose 2}.
}
By the Cauchy--Schwarz inequality, for each $j$,
\eq{
\sum_{D\in\dm_j} n_D^2 &\ge \frac{1}{|\dm_j|}\biggl(\sum_{D\in \dm_j} n_D\biggr)^2 = \frac{n^2}{2^{dj}}. 
}
Thus,
\eq{
H_n(x_1,\ldots, x_n) &\ge \frac{n^2}{2}\sum_{j=1}^k 2^{-dj} - \frac{nk}{2} + {n\choose 2}= \frac{n^2}{2}\frac{1-2^{-dk}}{2^d-1}- \frac{nk}{2} + {n\choose 2}.
}
By our choice of $k$, this completes the proof.
\end{proof}

\subsection{Estimates for the partition function}
Recall that for a measurable function $f: \Sigma_n\to\rr$, its expected value under $\mu_{n,\beta}$ is denoted by $\mu_{n,\beta}(f)$.  %The expectation under the uniform distribution on $\Sigma_n$ will be denoted by $\ee(f)$. 
\begin{lmm}\label{znlmm2d}
There is a constant $C_2$ such that for any $n\ge 0$ and $\beta>0$,
\[
\exp\biggl(-\frac{2^d\beta n}{2^d-1} \biggr)\le \frac{Z(n+1, \beta)}{Z(n, \beta)}\le \exp\biggl(-\frac{2^d\beta n}{2^d-1} + C_2 \log (n+1)\biggr). 
\]
\end{lmm}
\begin{proof}
The proof of Lemma \ref{znlmm} goes through verbatim, the only change being that we need to use Theorem \ref{lnlowthm2d} instead of Theorem \ref{lnlowthm}.\end{proof}
%Lemma \ref{znlmm} is iterated to obtain the following corollary.
\begin{cor}\label{zncor2d}
For any $n\ge 0$, $\beta>0$, and any $k\ge -n$,
\[
\frac{Z(n+k, \beta)}{Z(n,\beta)} \le \exp\biggl(-\frac{2^d\beta nk}{2^d-1} -\frac{2^d\beta k(k-1)}{2(2^d-1)} + C_2 \beta |k| \log(n+|k|+1)\biggr),
\]
where $C_5$ is the constant from Lemma \ref{znlmm2d}.
\end{cor}
\begin{proof}
Again, the proof of Corollary \ref{zncor} goes through verbatim, except that we need to use Lemma \ref{znlmm2d} instead of Lemma \ref{znlmm}.
\end{proof}
\subsection{Proofs of the upper bounds}\label{proof2d}
Let us now fix some $n\ge 0$ and $\beta>0$. In the following, $(X_1,\ldots, X_n)$ will denote a random configuration drawn from the measure $\mu_{n,\beta}$. We will assume that $(X_1,\ldots,X_n)$ is defined on some abstract probability space $(\Omega, \mf, \pp)$. Expectation, variance and covariance with respect to $\pp$ will be denoted by $\ee$, $\var$ and $\cov$ respectively.
\begin{lmm}\label{varlmm2d}
Let $D_1,\ldots,D_{2^d}$ denote the $2^d$ elements of $\dm_1$, and for each $1\le i\le 2^d$, let $N_i := |\{j: X_j\in D_i\}|$. Then for each $i$, $\ee(N_i)= n/2^d$ and 
\[
\var(N_i)\le K(\beta) (\log (n+1))^2,
\]
where $K(\beta)$ is a non-increasing function of $\beta$.
\end{lmm}
\begin{proof}
We have already defined universal constants $C_1$ and $C_2$ in the previous subsections. In this proof, we will denote further universal constants by $C_3, C_4,\ldots$ without explicitly mentioning that they denote universal constants on each occasion.

The identity $\ee(N_i)=n/2^d$ follows by symmetry. We will now prove the claimed bound on the variance. The cases $n=0$ and $n=1$ are trivial, so assume that $n\ge 2$. As in the proof of Lemma \ref{varlmm}, we have a recursion for the partition function, although the recursion is slightly different due to the different nature of the potential:
\eq{
&Z(n,\beta) \\
&= \sum_{\substack{0\le n_1,\ldots, n_{2^d}\le n\\ n_1+\cdots +n_{2^d}=n}} \frac{n!}{n_1!n_2!\cdots n_{2^d}!} e^{-\beta\sum_{1\le i<j\le 2^d} n_i n_j}\prod_{i=1}^{2^d} (2^{-dn_i}Z(n_i,\beta)e^{-\beta {n_i\choose 2}})\\
&= \sum_{\substack{0\le n_1,\ldots, n_{2^d}\le n\\ n_1+\cdots +n_{2^d}=n}} \frac{2^{-dn}e^{-\beta {n\choose 2}}n!}{n_1!n_2!\cdots n_{2^d}!} \prod_{i=1}^{2^d} Z(n_i,\beta).
}
Moreover, for any $(n_1,\ldots, n_{2^d})$ occurring in the above sum,
\eq{
\pp(N_1=n_1,\ldots, N_{2^d}=n_{2^d}) &= \frac{2^{-dn}e^{-\beta {n\choose 2}}n!}{n_1!n_2!\cdots n_{2^d}!} \frac{\prod_{i=1}^{2^d} Z(n_i,\beta) }{Z(n,\beta)}.
}
Choose nonnegative integers $m_1,\ldots, m_{2^d}$ such that $m_1+\cdots+m_{2^d}=n$ and $|m_i-n/2^d|\le 1$ for each $i$. For convenience, let
\eq{
f(n_1,\ldots, n_{2^d}) := \frac{n!}{n_1!n_2!\cdots n_{2^d}!}, \ \ \ h(n_1,\ldots, n_{2^d}) := \prod_{i=1}^{2^d} Z(n_i,\beta).
}
Take any $k_1,\ldots,k_{2^d}\in \zz$ such that $k_1+\cdots+k_{2^d}=0$ and $0\le m_i+k_i\le n$ for each~$i$. Then by Corollary \ref{zncor2d},
\eq{
&\frac{h(m_1+k_1,\ldots, m_{2^d}+k_{2^d})}{h(m_1,\ldots, m_{2^d})} \\
&\le \prod_{i=1}^{2^d} \exp\biggl(-\frac{2^d\beta m_ik_i}{2^d-1} -\frac{2^d\beta k_i(k_i-1)}{2(2^d-1)} + C_2 \beta |k_i| \log(n+|k_i|+1)\biggr)\\
&\le \prod_{i=1}^{2^d} \exp\biggl(-\frac{2^d\beta (nk_i/2^d - |k_i|)}{2^d-1} -\frac{2^d\beta k_i(k_i-1)}{2(2^d-1)} + 2C_2\beta |k_i|\log n\biggr)\\
&\le \exp\biggl(-\frac{2^d\beta}{2(2^d-1)}\sum_{i=1}^{2^d} k_i^2 + C_3\beta \log n\sum_{i=1}^{2^d} |k_i|\biggr).
}
Therefore,
\eq{
&\frac{\pp(N_1=m_1+k_1,\ldots, N_{2^d}=m_{2^d}+k_{2^d})}{\pp(N_1=m_1,\ldots, N_{2^d}=m_{2^d})}\\
&\le \frac{f(m_1+k_1,\ldots, m_{2^d}+k_{2^d})}{f(m_1,\ldots, m_{2^d})} \exp\biggl(-\frac{2\beta}{3}\sum_{i=1}^{2^d} k_i^2 + C_3 \beta \log n \sum_{i=1}^{2^d}|k_i|\biggr).}
This shows that there are positive constants $C_4$ and $C_5$ such that if 
\[
\max_{1\le i\le 2^d} |k_i|\ge C_4\log n,
\]
then 
\eq{
&\frac{\pp(N_1=m_1+k_1,\ldots, N_{2^d}=m_{2^d}+k_{2^d})}{\pp(N_1=m_1,\ldots, N_{2^d}=m_{2^d})}\\
&\le \frac{f(m_1+k_1,\ldots, m_{2^d}+k_{2^d})}{f(m_1,\ldots, m_{2^d})} e^{-C_5\beta (\log n)^2}. %\label{fineq}
}
It is now easy to complete the proof by imitating the last part of the proof of Lemma \ref{varlmm}.
\end{proof}
For a Borel set $A\subseteq [0,1)^d$, let $X(A)$ and $N(A)$ be defined as before. Also, define $\{\mf_k\}_{k\ge 0}$ as before.
\begin{lmm}\label{condlmm2d}
Conditional on $\mf_k$, the random sets $\{X(D): D\in \dm_k\}$ are mutually independent. Moreover, for any $D\in \dm_k$, conditional on $\mf_k$, $X(D)$ has the same distribution as a scaled version of a point process from the measure $\mu_{N(D), \beta}$.
\end{lmm}
\begin{proof}
The proof is the same as the proof of Lemma \ref{condlmm}, except that $\beta$ need not be replaced by $2^k \beta$ due to the different nature of the potential.
\end{proof}
%Lemma \ref{condlmm} allows us to compute conditional means and variances.
\begin{lmm}\label{condcalc2d}
If $D\in \dm_k$ and $D'$ is a child of $D$, then  
\[
\ee(N(D')|\mf_k) = \frac{N(D)}{2^d}
\]
and
\eq{
\var(N(D')|\mf_k) &\le K(\beta) (\log (N(D)+1))^{2},
}
where $K$ is the function from Lemma \ref{varlmm2d}.
\end{lmm}
\begin{proof}
The formula for the conditional expectation follows from Lemma~\ref{condlmm} and symmetry, and the bound on the conditional variance follows from Lemma \ref{condlmm2d} and Lemma \ref{varlmm2d}.
\end{proof}
\begin{lmm}\label{expvarlmm2d}
For any $D\in \dm$, $\ee(N(D))=\leb(D)n$ and 
\[
\var(N(D)) \le C(\beta) (\log(2^d\leb(D)n+3))^2,
\]
where $C(\beta)$ depends only on $\beta$.
\end{lmm}
\begin{proof}
Suppose that $D\in \dm_k$. The formula for the expectation follows easily by iterating the formula for the conditional expectation from Lemma \ref{condcalc2d}, and observing that $\leb(D)=2^{-dk}$. Next, let $D'$ be the parent of $D$. Then by Lemma \ref{condcalc2d}, the formula for expected value, and the concavity of the map $x\mapsto (\log (x+3))^2$ on the nonnegative axis,
\eq{
\ee(N(D)^2) &= \ee(N(D)^2 - (\ee(N(D)|\mf_{k-1}))^2) + \ee((\ee(N_{D}|\mf_{k-1}))^2)\\
&= \ee(\var(N(D)|\mf_{k-1})) + 2^{-2d} \ee(N(D')^2)\\
&\le K(\beta) \ee((\log (N(D')+1))^{2}) + 2^{-2d} \ee(N(D')^2)\\
&\le K(\beta) \ee((\log (N(D')+3))^{2}) + 2^{-2d} \ee(N(D')^2)\\
&\le K(\beta) (\log\ee(N(D')+3))^{2} + 2^{-2d} \ee(N(D')^2)\\
&= K(\beta) (\log(2^{-d(k-1)}n+3))^{2} + 2^{-2d}\ee(N(D')^2). 
}
Iterating this, we get
\eq{
\ee(N(D)^2) &\le  K(\beta)\sum_{r=0}^{k-1}(\log(2^{d+rd}\leb(D)n+3))^22^{-2rd} + 2^{-2dk}n^2.
}
Now note that for any $r\ge0$,
\eq{
\frac{\log(2^{d+rd}\leb(D)n+3)}{\log(2^{d}\leb(D)n+3)}&\le \frac{\log(2^{d}\leb(D)n+3) + \log 2^{rd}}{\log(2^{d}\leb(D)n+3)}\\
&= 1+  \frac{rd\log 2}{\log(2^{d}\leb(D)n+3)}\le 1+  \frac{rd\log 2}{\log3}.
}
Thus, 
\eq{
\ee(N(D)^2) &\le K(\beta)(\log(2^d\leb(D)n+3))^2\sum_{r=0}^\infty\biggl(1+\frac{rd\log 2}{\log 3}\biggr)^22^{-2rd} \\
&\qquad + 2^{-2dk}n^2\\
&\le C(\beta)(\log(2^d\leb(D)n+3))^2 + 2^{-2dk}n^2,
}
where $C(\beta)$ depends only on $\beta$. This completes the proof, since $\ee(N(D)) = \leb(D)n = 2^{-dk}n$.
\end{proof}
Now take any nonempty open set $U\subseteq [0,1)^d$  with regular boundary, and let $A(U)$ be defined as in \eqref{audef}. Define $\uu$, $\uu_j$, $\vv_j$ and $M_j$ as in the 3D case. It is easy to see that Lemma \ref{ulmm}, Corollary \ref{expcor} and Lemma \ref{martingale} remain valid in the 2D and 1D cases.
\begin{lmm}\label{mjvar2d}
For any $j\ge 1$ such that $\sqrt{d}\cdot 2^{-j+1}\le \diam(U)$, 
\[
\var(M_j)\le  C(\beta) A(U) (\log(2^{-d(j-1)} n+3))^{2} 2^{(d-1)j} + \var(M_{j-1}),
\]
where $C(\beta)$ is a constant that depends only on $\beta$. %, and $\area(\partial U)$ is the area of the boundary of $U$. %Consequently, $\var(M_j)\le C(\beta)S(U) jn^{2/3}$.
\end{lmm}
\begin{proof}
In this proof, $C(\beta)$ will denote any constant that depends only on $\beta$. Equations \eqref{varvar} and \eqref{mjmj} are still valid. 
If $D,D'\in \uu_j\cup\vv_j$ have different parents, then $N(D)$ and $N(D')$ are conditionally independent by Lemma~\ref{condlmm2d}, and hence the conditional covariance is zero. Otherwise, Lemma \ref{condcalc2d} and the Cauchy--Schwarz inequality imply that 
\[
|\cov(N(D), N(D')|\mf_{j-1})|\le C(\beta)(\log (N(D'')+1))^{2},
\]
where $D''$ is the parent of $D$ and $D'$. 
Thus, by Lemma \ref{expvarlmm2d} and the concavity of the map $x\mapsto (\log(x+3))^2$ on the nonnegative real axis,
\eq{
|\ee(\cov(N(D), N(D')|\mf_{j-1}))| &\le C(\beta) (\log(\leb(D'') n+3))^{2} \\
&= C(\beta) (\log(2^{-d(j-1)} n+3))^{2}.
}
On the other hand,  each $D\in \uu_j\cup \vv_j$ has at most $2^d-1$ siblings  that belong to $\uu_j\cup \vv_j$. Since $p(D)2^{-dj} = p(D)\leb(D) = \leb(D\cap U)$, this shows that 
\eq{
\ee(\var(M_j|\mf_{j-1})) &\le C(\beta) (\log(2^{-d(j-1)} n+3))^{2}\sum_{D\in \uu_j\cup \vv_j} 2^dp(D)\\
&= C(\beta) (\log(2^{-d(j-1)} n+3))^{2} 2^{d(j+1)}\sum_{D\in \uu_j\cup \vv_j}\leb(D\cap U).
%&\le 28K(\beta) n^{2/3}\leb(U).
}
Note that each element of 
\[
\bigcup_{D\in \uu_j\cup \vv_j}( D\cap U)
\]
is within distance $\sqrt{d}\cdot2^{-j+1}$ of $\partial U$. Since $\sqrt{d}\cdot 2^{-j+1}\le \diam(U)$, inequality~\eqref{audef} gives 
\[
\sum_{D\in \uu_j\cup \vv_j}\leb(D\cap U)\le A(U)\sqrt{d}\cdot 2^{-j+1}.
\]
Consequently,
\[
\ee(\var(M_j|\mf_{j-1})) \le C(\beta) A(U) (\log(2^{-d(j-1)} n+3))^{2} 2^{(d-1)j},
\]
where $C(\beta)$ depends only on $\beta$. 
The proof is completed by plugging this bound into \eqref{varvar}. %For the second assertion, observe that $M_0$ is a constant and hence $\var(M_0)=0$. 
\end{proof}
We now have all the ingredients for proving the following analog of Theorem \ref{mainthm}.
\begin{thm}[Hyperuniformity at all scales in 2D and 1D]\label{mainthm2d}
Let $U$ and $N(U)$ be as in Theorem~\ref{macthm}. Suppose that $\diam(U)\ge n^{-1/d}$. Let $A(U)$ be the constant defined in \eqref{audef}. Then 
\[
\ee(N(U)) = \leb(U)n
\]
and 
\eq{
\var(N(U))&\le C(\beta)(A(U)n^{(d-1)/d}+1) (\log(7\diam(U)^d n))^2,
}
where $C(\beta)$ is a constant that depends only on $\beta$. %, $\area(\partial U)$ is the area of the boundary of $U$, $\diam(U)$ is the diameter of $U$, and $\vol(U)$ is the volume of $U$.
\end{thm}
\begin{proof}
Throughout this proof, $C(\beta)$ will denote any constant that depends only on $\beta$. The value of $C(\beta)$ may change from line to line or even within a line. 

The formula for the expectation follows from the $d$-dimensional version of Corollary \ref{expcor}. It remains to prove the variance bound. Choose $k$ such that
\[
\frac{1}{2}n^{-1/d}\le \sqrt{d} \cdot 2^{-k}\le n^{-1/d}. 
\]
Equation \eqref{numk} remains valid, as does the inequality
\eq{
\var(N(U)|\mf_k) &\le  \sum_{D\in \vv_k} p(D) N(D)^2. 
}
By Lemma \ref{expvarlmm2d} and our choice of $k$, 
\eq{
\ee(N(D)^2)&\le C(\beta) (\log(2^d\leb(D)n+3))^2 + \leb(D)^2n^2\le C(\beta)
}
for all $D\in \vv_k$. Note that each element of 
\[
\bigcup_{D\in \vv_k} (D\cap U) 
\]
is within distance $\sqrt{d}\cdot 2^{-k}$ of $\partial U$, and $p(D) 2^{-dk} = \leb(D\cap U)$. Since $\sqrt{d}\cdot 2^{-k}\le n^{-1/d}\le \diam(U)$ by our choice of $k$, this gives
\eq{
\ee(\var(N(U)|\mf_k)) &\le C(\beta)2^{dk}\sum_{D\in \vv_k} \leb(D\cap U) \\
&\le C(\beta) 2^{dk}A(U) 2^{-k}\le C(\beta)A(U)n^{(d-1)/d}.
}
Let $l$ be the smallest integer such that $\sqrt{d}\cdot 2^{-l} \le \diam(U)$. Note that $l\le k$. 
Together with \eqref{numk} and Lemma \ref{mjvar2d}, the above inequality shows that
\eq{
&\var(N(U))\le C(\beta)A(U)\sum_{j=l+1}^k(\log(2^{-d(j-1)} n+3))^{2} 2^{(d-1)j}  + \var(M_l)\\
&\le  C(\beta)A(U)(\log(2^d\diam(U)^d n+3))^{2}\sum_{j=l+1}^k 2^{(d-1)j}  + \var(M_l)\\
&\le C(\beta)A(U)n^{(d-1)/d} (\log(2^d\diam(U)^d n+3))^{2}  + \var(M_l).
}
By the definition of $l$, $\uu_i$ if empty for all $i<l$. Therefore
\[
M_l = \sum_{D\in \uu_l\cup \vv_l} p(D)N(D). 
\]
Note that for any $D\in \uu_l\cup \vv_l$, Lemma \ref{expvarlmm2d} gives
\eq{
\var(p(D)N(D)) &= p(D)^2\var(N(D))\le C(\beta)(\log(2^d\leb(D)n+3))^2\\
&= C(\beta)(\log(2^d 2^{-dl} n + 3))^2\\
&\le C(\beta)(\log(2^d \diam(U)^d n + 3))^2.
}
Moreover, it is easy to see that $U$ intersects at most $2^d$ members of $\dm_l$, and therefore $|\uu_l \cup \vv_l|\le 2^d$. From these observations, we get
\eq{
\var(M_l)\le C(\beta)(\log(2^d \diam(U)^d n + 3))^2.
} 
This completes the proof of the theorem. 
\end{proof}

\begin{proof}[Proofs of Theorems \ref{macthm2d} and \ref{micthm2d}]
These are consequences of Theorem \ref{mainthm2d} in the same way as Theorems \ref{macthm} and \ref{micthm} followed from Theorem \ref{mainthm}. 
\end{proof}

Finally, let us prove Theorem \ref{linearthm2d}. 

\begin{proof}[Proof of Theorem \ref{linearthm2d}]
The proof is very similar to the proof of Theorem~\ref{linearthm}, with minor modifications. As usual, $C(\beta)$ denotes any constant that depends only on $\beta$. Define $f(D)$ and $W_k$ as in the proof of Theorem \ref{linearthm}. Then $W_k$ is again a martingale, and equation \eqref{xfk} is still valid. 
Now choose $k$ such that
\[
n^{-1/d}\le 2^{-k}\le 2n^{-1/d}. 
\]
Then \eqref{varxf} continues to hold. 
Take any $j$. For each $D\in \dm_{j-1}$, let $c(D)$ denote the set of $2^d$ children of $D$. Proceeding as in the proof of Theorem~\ref{linearthm}, we get 
\eq{
&\var(X(f_j)|\mf_{j-1}) \\
&= \sum_{D\in \dm_{j-1}} \ee\biggl(\biggl(\sum_{D'\in c(D)} f(D')N(D') - f(D)N(D)\biggr)^2\biggl|\mf_{j-1}\biggr).
} 
Now notice that for any $D\in \dm_{j-1}$, 
\eq{
&\sum_{D'\in c(D)} f(D')N(D') - f(D)N(D) \\
&= \sum_{D'\in c(D)} (f(D')- f(D))\biggl(N(D')-\frac{N(D)}{2^d}\biggr).
}
Recall that $L$ is the Lipschitz constant of $f$. For any $D'\in c(D)$,
\[
|f(D')-f(D)|\le \sqrt{d} L 2^{-j+1}.
\]
As in the proof of Theorem \ref{linearthm}, 
\eq{
&\biggl(\sum_{D'\in c(D)} (f(D')- f(D))\biggl(N(D')-\frac{N(D)}{2^d}\biggr)\biggr)^2 \\
&\le 4^{-j+3}L^2\sum_{D'\in c(D)} \biggl(N(D')-\frac{N(D)}{2^d}\biggr)^2.
}
Therefore, by Lemma \ref{condcalc2d}, 
\eq{
&\ee\biggl(\biggl(\sum_{D'\in c(D)} f(D')N(D') - f(D)N(D)\biggr)^2\biggl|\mf_{j-1}\biggr)\\
&\le 4^{-j+3} L^2 \sum_{D'\in c(D)}\var(N(D')|\mf_{j-1})\\
&\le 4^{-j+4} L^2 K(\beta)(\log(N(D)+1))^2\le 4^{-j+4}L^2K(\beta)(\log(n+1))^2.
}
Consequently,
\eq{
\ee(\var(X(f_j)|\mf_{j-1})) &\le C(\beta) L^2(\log n)^24^{-j}|\dm_{j-1}|\le C(\beta)L^22^{(d-2)j} (\log n)^2.
}
For $D\in \dm_k$, let $s(D)$ be defined as in the proof of Theorem \ref{linearthm}. Then as before, we have 
\eq{
\var(X(f)|\mf_k) &\le \sum_{D\in \dm_k}\ee((s(D)-f(D)N(D))^2|\mf_k). 
}
By the Lipschitz condition,
\eq{
|s(D)-f(D)N(D)|\le \sqrt{d}L 2^{-k}N(D)
}
for each $D\in \dm_k$. Thus, by Lemma \ref{expvarlmm2d} and our choice of $k$,
\eq{
\ee((s(D)-f(D)N(D))^2) &\le 4^{-k+1}L^2 \ee(N(D)^2)\le C(\beta)L^2 4^{-k}.
}
Consequently,
\eq{
\ee(\var(X(f)|\mf_k)) &\le C(\beta)L^24^{-k}|\dm_k|\le C(\beta) L^22^{(d-2)k}\le C(\beta)L^2 n^{(d-2)/d}.
}
The proof is now easily completed by combining the steps. 
\end{proof}

\subsection{Proofs of the lower bounds}
Let us now prove Theorem \ref{lowthm2d}. The proof is similar to the proof of Theorem \ref{lowthm}, but with some significant changes due to the different nature of the potential. Let 
\[
\mt:= \{z+[0,1)^2: z\in \zz^2\}. 
\] 
%The following lemma is the two-dimensional analog of Lemma \ref{geomlmm0}.
\begin{lmm}\label{geomlmm02d}
Let $\mt$ be as above. Take any $D\in \mt$ and any $x\in D$. Let $\delta$ be the distance of $x$ from the boundary of $D$. Then any line through $x$ bifurcates $D$ into two parts, each of which has volume at least $\pi \delta^2/2$. 
\end{lmm}
\begin{proof}
Same as the proof of Lemma \ref{geomlmm0}. 
\end{proof}
%The second lemma is an easy fact about intervals.
\begin{lmm}\label{geomlmm2d}
Take any $x\in \rr^2$ and a unit vector $u = (u_1,u_2)\in S^1$. Let $L$ be the line that contains $x$ and is perpendicular to $u$. Suppose that
\eeq{
\min\{|u_1|, |u_2|\} \ge 0.1.\label{u1u2u32d}
}
Then there is an element $D\in \mathcal{T}$, within Euclidean distance $\sqrt{101}$ from $x$, which  is bifurcated by the line $P$ in such a way that each part has volume at least $6\times 10^{-5}$.
\end{lmm}
\begin{proof}
Take any $x = (x_1,x_2)\in \rr^2$ and $u=(u_1,u_2)\in S^1$ as in the statement of the lemma.  Let $L_0$ be the line with normal vector $u$ that contains the origin. Define 
\eq{
y_1= \textup{sign}(u_1), \ \ y_2 = -\frac{|u_1|}{u_2}. 
}
Then $y= (y_1,y_2)\in L_0$. Also, we have $|y_1|=1$,  and by condition~\eqref{u1u2u32d} and the fact that $|u_2|\le 1$,
\[
|y_2| = \frac{|u_1|}{|u_2|} \ge |u_1|\ge0.1.
\]
Now consider the set
\[
I_1 = \{x_1 + \alpha y_1: 0\le \alpha\le 1\}.
\]
Since $|y_1|=1$, $I_1$ is an interval of length $1$. By Lemma \ref{intlmm}, $I_1$ has a subinterval of $I_2$ of length $0.25$ such that any integer is at least at a distance $0.25$ from $I_2$. Moreover, since $|y_1|= 1$, $I_2$ is of the form
\[
\{x_1+\alpha y_1: a\le \alpha\le b\},
\]
where $b-a= 0.25$. Let
\[
I_3 := \{x_2+\alpha y_2: a\le \alpha\le b\}. 
\]
Since $|y_2|\ge 0.1$, $I_3$ has length at least $0.025$. Thus by Lemma \ref{intlmm}, $I_3$ contains a subinterval $I_4$ of length $0.00625$ such that any integer is at a distance at least $0.00625$ from $I_4$.

In particular, there is some $\alpha\in [0,1]$ such that $x_1+\alpha y_1\in I_2$ and $x_2+\alpha y_2\in I_4$. The distance of $x_i+\alpha y_i$ from the nearest integer is at least $0.00625$ for each $i$. Thus, the distance of the point $x+\alpha y$ from the boundary of the square $D\in \mt$ that contains $x+\alpha y$ is at least $0.00625$. By Lemma \ref{geomlmm02d} and the fact that $x+\alpha y\in L$, this proves $L$ bifurcates $D$ into two parts, each of which has volume at least $6\times 10^{-5}$. Lastly, note that 
\eq{
|(x+\alpha y)-x|&\le |y| = \sqrt{y_1^2+y_2^2}\le \sqrt{1 + \frac{1}{0.1^2}}\le \sqrt{101},
}
since $|u_1|\le 1$ and $|u_2|\ge 0.1$. 
This completes the proof of the lemma. 
\end{proof}
Now recall that the boundary of the set $U$ in the statement of Theorem~\ref{lowthm2d} is a simple smooth closed curve. In particular, we can choose a unit normal vector $u(x)$ at each $x\in \partial U$ such that the map $x\mapsto u(x)$ is smooth.
\begin{lmm}\label{dcapu2d}
Take any $x\in \partial U$ such that the normal vector $u(x)$ satisfies~\eqref{u1u2u32d}. Then there is some $j_0$ depending only on $U$ (but not on $x$), such that for all $j\ge j_0$, there is some $D\in \dm_j$ at distance at most $\sqrt{101} \cdot 2^{-j}$ from $x$, such that
\eeq{
10^{-5}\le \frac{\leb(D\cap U)}{\leb(D)}\le 1- 10^{-5}. \label{djbd2d}
}
\end{lmm}
\begin{proof}
Same as the proof of Lemma \ref{dcapu}, using Lemma \ref{geomlmm2d} instead of Lemma \ref{geomlmm}.
\end{proof}
%The above lemma leads to the following result, which is a key component of the proof of Theorem \ref{lowthm}.
\begin{lmm}\label{k1lmm2d}
There is some $K_1>0$ and some $j_1\ge 1$ depending only on $U$ such that for any $j\ge j_1$, there is a set of at least $K_1 2^j$ squares $D\in \dm_j$ that satisfy~\eqref{djbd2d} and the union of these squares has diameter at most $\diam(U)/3$.
\end{lmm}
\begin{proof}
Same as the proof of Lemma \ref{k1lmm}, with a small adjustment for dimension that replaces $K_14^j$ by $K_12^j$. 
%Let $L$ be the line through the origin that is perpendicular to the vector $(1,1)$. Let $\alpha_0$ be the largest $\alpha$ such that the line $L_\alpha := (\alpha,\alpha)+L$ intersects the closure of $U$. Let $x$ be a point of intersection. Then $x\in \partial U$, and $L_{\alpha_0}=T_x$, where $T_x$ is the tangent at $x$. Consequently, there is some $0<\ep<\diam(U)/3$  such that for every $y\in B(x,\ep)\cap \partial U$, $u(y)$ satisfies~\eqref{u1u2u32d}. Due to the boundedness of the curvature of $\partial U$, a small enough choice of $\ep$ guarantees that for any $\delta\in (0,1)$, there are at least $C\delta^{-1}$ points in $B(x,\ep)\cap \partial U$ (where $C$ is a positive constant that depends only on $U$ and $B(x,\ep)$ is the disk of radius $\ep$ around $x$) such that any two points are at distance at least $50\delta$ from each other. The proof is completed by taking  $\delta=2^{-j}$ and applying Lemma~\ref{dcapu2d}.
\end{proof}
%Lastly, we need  a lemma about our point process. Recall that for any $D\in \dm$, $N(D)$ is the number of points landing in $D$.
\begin{lmm}\label{repulsion2d}
Take any $n\ge 1$ and $\beta>0$, and a Borel set $A\subseteq [0,1)^2$ with $0<\leb(A)<1$. Let $\delta>0$ be a number such that $\delta \le \leb(A)\le 1-\delta$. Then
\[
c(\beta, n,\delta)\le \var(N(A)) \le n^2,
\]
where $c(\beta, n,\delta)$ is a positive real number that depends only on $\beta$, $n$ and $\delta$. 
\end{lmm}
\begin{proof}
The upper bound is trivial since $N(A)\le n$. For the lower bound, the case $n=1$ is easy, since in that case $N(A)$ is a Bernoulli$(\leb(A))$ random variable. So let us take $n\ge 2$. Trivially, $Z(n,\beta)\le 1$. Therefore, by Jensen's inequality and Lemma \ref{wlmm2d},
\eq{
\pp(N(A)=n)&=\mu_{n,\beta}(A^n) \ge \int_{A^n} e^{-\beta H_n(x_1,\ldots,x_n)}\, dx_1\cdots dx_n \\
&\ge \leb(A^n) \exp\biggl(-\frac{\beta}{\leb(A^n)} \int_{A^n}H_n(x_1,\ldots,x_n)\, dx_1\cdots dx_n \biggr)\\
&\ge \leb(A^n)\exp\biggl(-\frac{4\beta }{3\leb(A^n)}{n \choose 2}\biggr)\\
&\ge \delta^n\exp\biggl(-\frac{4\beta }{3\delta^n}{n \choose 2}\biggr). 
}
Similarly, if $B= [0,1)^2\setminus A$, then 
\eq{
\pp(N(A)=0)=\mu_{n,\beta}(B^n)\ge \delta^n\exp\biggl(-\frac{4\beta }{3\delta^n}{n \choose 2}\biggr).
}
With the two lower bounds derived above, it is now easy to complete the proof, for example using Chebychev's inequality.
\end{proof}
%Finally, we are ready to prove Theorem \ref{lowthm}. Recall the filtration $\{\mf_k\}_{k\ge 0}$ defined earlier.
\begin{proof}[Proof of Theorem \ref{lowthm2d}]
In this proof, as in the proof of Theorem \ref{lowthm}, the phrase `$n$ sufficiently large' will mean `$n\ge n_0$, where $n_0$ depends only on $U$ and $\beta$'. Also, $C$ will denote any positive universal constant, $C(\beta)$ will denote any positive constant that depends only on $\beta$, and $C(U,\beta)$ will denote any positive constant that depends only on $U$ and~$\beta$.

Choose $k$ such that 
\[
n^{-1/2}\le 2^{-k} \le 2n^{-1/2}. 
\]
Then for any $D\in \dm_k$, Lemma \ref{expvarlmm2d} gives
\eeq{
\ee(N(D)^2) \le L_1(\beta),\label{nd22d}
}
where $L_1(\beta)$ is a positive integer that depends only on $\beta$. Let 
\[
m := 1000 L_1(\beta).
\]
If $n$ is sufficiently large, then there is a set $\mc\subseteq \dm_k$ that satisfies the conclusions of Lemma \ref{k1lmm2d}. In particular, arguing as in the proof of Theorem \ref{lowthm}, we get
\eeq{
C_1(U,\beta) 2^k\le |\mc|\le C_2(U,\beta)2^k,\label{mcsize2d}
}
where $C_1(U,\beta)$ and $C_2(U,\beta)$ are positive constants that depend only on $U$ and $\beta$. 
Let $Q$ be the union of the elements of $\mc$.  Proceeding as in the proof of Theorem \ref{lowthm}, and using Theorem \ref{mainthm2d} instead of Theorem \ref{mainthm}, we get
\eeq{
\var(N(Q)) &\le C(U,\beta) n^{1/2}(\log n)^2, \label{varnq2d}
}
provided that $n$ is sufficiently large. 
Also, by Lemma \ref{expvarlmm2d} and our choice of~$k$,
\eq{
\ee(N(Q)) &= \leb(Q) n = |\mc| 4^{-k} n\ge |\mc|.
}
Thus, by \eqref{mcsize2d}, \eqref{varnq2d} and Chebychev's inequality,
\eeq{
\pp\biggl(\frac{N(Q)}{|\mc|}\ge \frac{1}{2}\biggr) &\ge 1-\frac{4\var(N(Q))}{|\mc|^2}\ge1-C(U,\beta) n^{-1/2}(\log n)^2.\label{a1ineq2d}
}
Let $a_1$, $a_2$, $p_1$, $p_2$ and $q$ be defined as in the proof of Theorem \ref{lowthm}. By the inequality \eqref{nd22d}, $\ee(a_2)\le L_1(\beta)$. Thus, 
\eeq{
\pp(a_2\ge 2L_1(\beta))\le \frac{1}{2}. \label{a2ineq2d}
}
By the Paley--Zygmund second moment inequality,
\eq{
p_1\ge \frac{a_1^2}{a_2}, 
}
and so by \eqref{a1ineq2d} and \eqref{a2ineq2d},
\eq{
\pp\biggl(p_1 \ge \frac{1}{8L_1(\beta)}\biggr) &\ge \pp\biggl(a_1\ge \frac{1}{2}, \, a_2 \le 2L_1(\beta)\biggr)\ge \frac{1}{2}-C(U,\beta) n^{-1/2}(\log n)^2. 
}
Choose $n$ so large that the above lower bound at least $1/3$. Next, note that by Lemma \ref{expvarlmm2d} and Markov's inequality,
\eq{ 
\ee(p_2) &\le \frac{1}{m|\mc|} \sum_{D\in \mc}\ee(N(D))\le \frac{4}{m},
}
and hence
\eq{
\pp\biggl(p_2\ge \frac{16}{m}\biggr)\le \frac{1}{4}.
}
Since $q=p_1-p_2$ and 
\[
\frac{1}{8L_1(\beta)} \ge \frac{32}{m},
\]
we get 
\eq{
\pp\biggl(q\ge \frac{16}{m}\biggr) &\ge \pp\biggl(p_1\ge \frac{32}{m},\, p_2\le \frac{16}{m}\biggr)\ge \frac{1}{3}-\frac{1}{4}= \frac{1}{12}.%\label{qeq2d}
}
Let $\mc_0$ be the set of all $D\in \mc$ such that $0<N(D)\le m$. The above inequality and \eqref{mcsize2d} show that if $n$ is sufficiently large, then 
\eeq{
\pp(|\mc_0| \ge L_2(\beta)n^{1/2})\ge \frac{1}{13},\label{cstar2d}
}
where  $L_2(\beta)$ is a positive constant that depends only on $\beta$. 
By Lemma~\ref{condlmm2d}, the random variables $\{N(D\cap U): D\in \dm_k\}$ are independent given $\mf_k$. Moreover, for each $D\in \mc_0$, $N(D\cap U)\le m\le C(\beta)$, and by Lemma \ref{condlmm2d} and Lemma \ref{repulsion2d},
\[
\var(N(D\cap U)|\mf_k) \ge L_3(U,\beta),
\]
where $L_3(U,\beta)$ is a positive constant that depends only on $U$ and $\beta$. (This is the crucial difference with the proof of Theorem \ref{lowthm}. The scale invariance of the model in dimension two is not valid in dimension three.) Thus, if we let
\[
M := \sum_{D\in \mc_0} N(D\cap U),
\]
then the Berry--Esseen theorem shows that for any interval $I$,
\eeq{
\pp(M\in I|\mf_k)\le \frac{C(U,\beta)(|I|+1)}{\sqrt{|\mc_0|}}, \label{berry2d}
}
where $|I|$ denotes the length of $I$. Since 
\[
N(U) = \sum_{D\in \dm_k} N(D\cap U) = \sum_{D\in \dm_k\setminus \mc_0}N(D\cap U) + M,
\]
and the two terms in the last expression are independent given $\mf_k$, the inequality \eqref{berry2d} implies that
\eq{ 
\pp(N(U)\in I|\mf_k)\le \frac{C(U,\beta)(|I|+1)}{\sqrt{|\mc_0|}}.
}
Therefore by \eqref{cstar2d},
\eq{
\pp(N(U)\in I)\le C(U,\beta)(|I|+1)n^{-1/4} + \frac{12}{13}
}
if $n$ is sufficiently large. This completes the proof.
\end{proof}

\section*{Acknowledgments}
I thank Erik Bates for carefully checking the proofs, and Paul Bourgade, Persi Diaconis, Subhro Ghosh, Adrien Hardy, Joel Lebowitz, Satya Majumdar, Charles Radin, Sylvia Serfaty and H.-T.~Yau for helpful discussions and comments. %I also thank the three anonymous referees for useful suggestions. 


\begin{thebibliography}{99}
\bibitem[Aizenman and Martin(1980)]{am80} {\sc Aizenman, M.} and {\sc Martin, P.~A.} (1980). Structure of Gibbs states of one-dimensional Coulomb systems. {\it Comm. Math. Phys.,} {\bf 78} no. 1, 99--116.
 
 
\bibitem[Ameur, Hedenmalm and Makarov(2011)]{ahm11} {\sc Ameur, Y., Hedenmalm, H.} and {\sc Makarov, N.} (2011). Fluctuations of eigenvalues of random normal matrices. {\it Duke Math. J.,} {\bf 159} no. 1, 31--81.

\bibitem[Ameur, Hedenmalm and Makarov(2015)]{ahm15} {\sc Ameur, Y., Hedenmalm, H.} and {\sc Makarov, N.} (2011). Random normal matrices and Ward identities. {\it Ann. Probab.,} {\bf 43} no. 3, 1157--1201.

\bibitem[Anderson, Guionnet and Zeitouni(2010)]{agz10} {\sc Anderson, G.~W.,  Guionnet, A.} and {\sc Zeitouni, O.} (2010). {\it An introduction to random matrices.} Cambridge University Press, Cambridge.

\bibitem[Bardenet and Hardy(2016)]{bardenethardy16} {\sc Bardenet, R.} and {\sc Hardy, A.} (2016). Monte Carlo with Determinantal Point Processes. {\it arXiv preprint arXiv:1605.00361.}

\bibitem[Bauerschmidt,  Bourgade, Nikula and Yau(2015)]{bbny15} {\sc Bauerschmidt,  R., Bourgade, P., Nikula, M.} and {\sc Yau, H.-T.} (2015).  Local density for two-dimensional one-component plasma. {\it arXiv preprint arXiv:1510.02074.}


\bibitem[Bauerschmidt,  Bourgade, Nikula and Yau(2016)]{bbny16} {\sc Bauerschmidt,  R., Bourgade, P., Nikula, M.} and {\sc Yau, H.-T.} (2016).  The two-dimensional Coulomb plasma: quasi-free approximation and central limit theorem. {\it arXiv preprint arXiv:1609.08582.}


\bibitem[Beck(1987)]{beck87} {\sc Beck, J.} (1987). Irregularities of distribution. I. {\it Acta Math.,} {\bf 159} nos. 1-2, 1--49. 

\bibitem[Bekerman, Lebl\'e and Serfaty(2017)]{bls17} {\sc Bekerman, F., Lebl\'e, T.} and {\sc Serfaty, S.} (2013). CLT for fluctuations of $\beta$-ensembles with general potential. {\it arXiv preprint arXiv:1706.09663.}


\bibitem[Bekerman and Lodhia(2016)]{bekermanlodhia16} {\sc Bekerman, F.}  and {\sc Lodhia, A.} (2016). Mesoscopic central limit theorem for general $\beta$-ensembles. {\it arXiv preprint arXiv:1605.05206.}


\bibitem[Ben Arous and Zeitouni(1998)]{bz98} {\sc Ben Arous, G.} and {\sc Zeitouni, O.} (1998). Large deviations from the circular law. {\it ESAIM Probab. Statist.,} {\bf 2}, 123--134.

\bibitem[Bendikov, Grigor'yan and Pittet(2014)]{bendikovetal12} {\sc Bendikov, A.~A., Grigor'yan, A.~A.} and {\sc Pittet, Ch.} (2012). On a class of Markov semigroups on discrete ultra-metric spaces. {\it Potential Anal.,} {\bf 37} no. 2, 125--169.

\bibitem[Bendikov, Grigor'yan, Pittet and Woess(2014)]{bendikovetal} {\sc Bendikov, A.~D., Grigor'yan, A.~A., Pittet, Ch.} and {\sc Woess, W.} (2014).  Isotropic Markov semigroups on ultra-metric spaces. (Russian) {\it Uspekhi Mat. Nauk,} {\bf 69} no. 4(418), 3--102; translation in {\it Russian Math. Surveys,}  {\bf 69} no. 4, 589--680.

\bibitem[Benfatto, Gallavotti and Nicol\`o(1986)]{benfattoetal86} {\sc Benfatto, G., Gallavotti, G.} and {\sc Nicol\`o, F.} (1986). The dipole phase in the two-dimensional hierarchical Coulomb gas: analyticity and correlations decay. {\it Comm. Math. Phys.,} {\bf 106} no. 2, 277--288.

\bibitem[Benfatto and Renn(1992)]{benfattorenn92} {\sc Benfatto, G.} and {\sc Renn, J.} (1992). Nontrivial fixed points and screening in the hierarchical two-dimensional Coulomb gas. {\it J. Stat. Phys.,} {\bf 67} no. 5, 957--980.

\bibitem[Berman(2014)]{berman14} {\sc Berman, R.~J.} (2014). Determinantal point processes and fermions on complex manifolds: large deviations and bosonization. {\it Comm. Math. Phys.,} {\bf 327} no. 1, 1--47.

\bibitem[Bolley, Chafa\"i and Fontbona(2017)]{bolleyetal17} {\sc Bolley, F., Chafa\"i, D.} and {\sc  Fontbona, J.} (2017).  Dynamics of a planar Coulomb gas. {\it arXiv preprint arXiv:1706.08776.}

\bibitem[Borodin, Gorin and Guionnet(2017)]{bgg17} {\sc Borodin, A., Gorin, V.} and {\sc Guionnet, A.} (2017). Gaussian asymptotics of discrete $\beta$-ensembles. {\it Publ. Math. Inst. Hautes \'Etudes Sci.,} {\bf 125}, 1--78.

\bibitem[Borodin and Sinclair(2009)]{borodinsinclair09} {\sc Borodin, A.} and {\sc Sinclair, C.~D.} (2009). The Ginibre ensemble of real random matrices and its scaling limits. {\it Comm. Math. Phys.,} {\bf 291} no. 1, 177--224.


\bibitem[Borot and Guionnet(2013)]{bg13} {\sc Borot, G.} and {\sc Guionnet, A.} (2013). Asymptotic expansion of $\beta$ matrix models in the one-cut regime. {\it Comm. Math. Phys.,} {\bf 317} no. 2, 447--483.

\bibitem[Borot and Guionnet(2013)]{bg13b} {\sc Borot, G.} and {\sc Guionnet, A.} (2013). Asymptotic expansion of $\beta$ matrix models in the multi-cut regime. {\it arXiv preprint arXiv:1303.1045.}

\bibitem[Borot, Guionnet and Kozlowski(2015)]{bgk15} {\sc Borot, G., Guionnet, A.} and {\sc Kozlowski, K.~K.} (2015). Large-$N$ asymptotic expansion for mean field models with Coulomb gas interaction. {\it Int. Math. Res. Not. IMRN,} {\bf 2015} no. 20, 10451--10524. 

\bibitem[Bourgade, Erd\H{o}s and Yau(2012)]{bey12} {\sc Bourgade, P., Erd\H{o}s, L.} and {\sc Yau, H.-T.} (2012). Bulk universality of general $\beta$-ensembles with non-convex potential. {\it J. Math. Phys.,} {\bf 53} no. 9, 095221, 19 pp.

\bibitem[Bourgade, Erd\H{o}s and Yau(2014)]{bey14} {\sc Bourgade, P., Erd\H{o}s, L.} and {\sc Yau, H.-T.} (2014). Universality of general $\beta$-ensembles. {\it Duke Math. J.,} {\bf 163} no. 6, 1127--1190. 

\bibitem[Bourgade, Erd\H{o}s and Yau(2014)]{bey14b} {\sc Bourgade, P., Erd\H{o}s, L.} and {\sc Yau, H.-T.} (2014). Edge universality of beta ensembles. {\it Comm. Math. Phys.,} {\bf 332} no. 1, 261--353.

\bibitem[Bourgade, Erd\H{o}s, Yau and Yin(2016)]{beyy16} {\sc Bourgade, P.,  Erd\H{o}s, L., Yau, H.-T.} and {\sc Yin, J.} (2016). Fixed energy universality for generalized Wigner matrices. {\it Comm. Pure Appl. Math.,} {\bf 69} no. 10, 1815--1881.

\bibitem[Bourgade, Yau and Yin(2014a)]{byy14a} {\sc Bourgade, P.,  Yau, H.-T.} and {\sc Yin, J.} (2014). Local circular law for random matrices. {\it Probab. Theory Related Fields,} {\bf 159} nos. 3-4, 545--595.


\bibitem[Bourgade, Yau and Yin(2014b)]{byy14b} {\sc Bourgade, P.,  Yau, H.-T.} and {\sc Yin, J.} (2014). The local circular law II: the edge case. {\it Probab. Theory Related Fields,} {\bf 159} nos. 3-4, 619--660.

\bibitem[Brascamp and Lieb(1975)]{bl75} {\sc Brascamp, H.~J.} and {\sc Lieb, E.~H.} (1975). Some inequalities for Gaussian measures and the long-range order of the one-dimensional plasma. {\it Functional integration and its applications,} 1--14. Clarendon Press, Oxford.

\bibitem[Breuer and Duits(2016)]{breuerduits16} {\sc Breuer, J.} and {\sc Duits, M.} (2016). Universality of mesoscopic fluctuations for orthogonal polynomial ensembles. {\it Comm. Math. Phys.,} {\bf 342} no. 2, 491--531.

\bibitem[Breuer and Duits(2017)]{breuerduits17} {\sc Breuer, J.} and {\sc Duits, M.} (2017). Central limit theorems for biorthogonal ensembles and asymptotics of recurrence coefficients. {\it J. Amer. Math. Soc.,} {\bf 30} no. 1, 27--66.

\bibitem[Castin(2006)]{castin} {\sc Castin, Y.} (2006). Basic theory tools for degenerate Fermi gases. {\it arXiv preprint arXiv:cond-mat/0612613.}

\bibitem[Chafa\"i, Gozlan and Zitt(2014)]{chafaietal14} {\sc Chafa\"i, D., Gozlan, N.} and {\sc Zitt, P.-A.} (2014). First-order global asymptotics for confined particles with singular pair repulsion. {\it Ann. Appl. Probab.,} {\bf 24} no. 6, 2371--2413.

\bibitem[Chafa\"i, Hardy and Ma\"ida(2016)]{chafai16} {\sc Chafa\"i, D., Hardy, A.,} and {\sc Ma\"ida, M.} (2016). Concentration for Coulomb gases and Coulomb transport inequalities. {\it arXiv preprint arXiv:1610.00980.}

\bibitem[Costin and Lebowitz(1995)]{costinlebowitz95} {\sc Costin, O.} and {\sc Lebowitz, J.} (1995). Gaussian Fluctuation in Random Matrices. {\it Phys. Rev. Lett.,} {\bf 75}, 69--72.

\bibitem[Deift(1999)]{deift99} {\sc Deift, P.~A.} (1999). {\it Orthogonal polynomials and random matrices: a Riemann-Hilbert approach.} American Mathematical Society, Providence, R.I.


\bibitem[Dhar, Kundu,  Majumdar, Sabhapandit and Schehr(2017)]{dharetal17} {\sc Dhar, A., Kundu, A., Majumdar, S.~N., Sabhapandit, S.} and {\sc Schehr, G.} (2017). Exact extremal statistics in the classical $1d$ Coulomb gas.  {\it Phys. Rev. Lett.,} {\bf 119}, 060601.

\bibitem[Diaconis and Evans(2001)]{diaconisevans01} {\sc Diaconis, P.} and {\sc Evans, S.~N.} (2001). Linear functionals of eigenvalues of random matrices. {\it Trans. Amer. Math. Soc.,} {\bf 353} no. 7, 2615--2633.

\bibitem[Dimock(1990)]{dimock90} {\sc Dimock, J.} (1990). The Kosterlitz-Thouless phase in a hierarchical model. {\it J. Phys. A,} {\bf 23} no. 7, 1207--1215.

\bibitem[Dyson(1953)]{dyson53} {\sc Dyson, F.~J.} (1953). The dynamics of a disordered linear chain. {\it Phys. Rev.,} {\bf 92} no. 6, 1331--1338.

\bibitem[Dyson(1969)]{dyson69} {\sc Dyson, F.~J.} (1969). Existence of a phase-transition in a one-dimensional Ising ferromagnet. {\it Comm. Math. Phys.,} {\bf 12} no. 2, 91--107.

\bibitem[Forrester(2010)]{forrester10} {\sc Forrester, P.~J.} (2010). {\it Log-gases and random matrices.}  Princeton University Press, Princeton, NJ.

\bibitem[Ghosh(2015)]{ghosh15} {\sc Ghosh, S.} (2015). Determinantal processes and completeness of random exponentials: the critical case. {\it Probab. Theory Related Fields,} {\bf 163} nos. 3-4, 643--665. 

\bibitem[Ghosh(2016)]{ghosh16} {\sc Ghosh, S.} (2016). Palm measures and rigidity phenomena in point processes. {\it Electron. Commun. Probab.,} {\bf 21}, Paper No. 85, 14 pp. 

\bibitem[Ghosh and Lebowitz(2017)]{gl17} {\sc Ghosh, S.} and  {\sc Lebowitz, J.} (2017).  Number rigidity in superhomogeneous random point fields. {\it J. Stat. Phys.,}  {\bf 166} no. 3-4, 1016--1027.

\bibitem[Ghosh and Lebowitz(2017)]{gl17b} {\sc Ghosh, S.} and {\sc Lebowitz, J.~L.} (2017). Fluctuations, large deviations and rigidity in hyperuniform systems: A brief survey. {\it Indian J. Pure Appl. Math.,} {\bf 48} no. 4, 609--631.

\bibitem[Ghosh and Peres(2017)]{ghoshperes17} {\sc Ghosh, S.} and {\sc Peres, Y.} (2017). Rigidity and Tolerance in point processes: Gaussian zeroes and Ginibre eigenvalues. {\it Duke Math. J.,} {\bf 166} no. 10, 1789--1858.


\bibitem[Ghosh and Zeitouni(2016)]{gz16} {\sc Ghosh, S.} and {\sc Zeitouni, O.} (2016). Large deviations for zeros of random polynomials with i.i.d. exponential coefficients. {\it Int. Math. Res. Not. IMRN,} {\bf 2016} no. 5, 1308--1347.

\bibitem[Ginibre(1965)]{ginibre65} {\sc Ginibre, J.} (1965). Statistical ensembles of complex, quaternion, and real matrices. {\it J. Math. Phys.,} {\bf 6}, 440--449.

\bibitem[Girko(1984)]{girko84} {\sc Girko, V.~L.} (1984). The circular law. {\it Teor. Veroyatnost. i Primenen.,} {\bf 29} no. 4, 669--679.


\bibitem[Guidi and Marchetti(2001)]{guidimarchetti01} {\sc Guidi, L.~F.} and {\sc Marchetti, D.~H.~U.} (2001). Renormalization Group Flow of the Two-Dimensional Hierarchical Coulomb Gas. {\it Comm. Math. Phys.,} {\bf 219} no. 3, 671--702.

\bibitem[Hardy(2012)]{hardy12} {\sc Hardy, A.} (2012). A note on large deviations for 2D Coulomb gas with weakly confining potential. {\it Electron. Commun. Probab.,} {\bf 17} no. 19, 12 pp. 

\bibitem[Holroyd and Soo(2013)]{holroydsoo13} {\sc Holroyd, A.~E.} and {\sc Soo, T.} (2013). Insertion and deletion tolerance of point processes. {\it Electron. J. Probab.,} {\bf 18} no. 74, 24 pp.

\bibitem[Hough, Krishnapur, Peres and Vir\'ag(2009)]{houghetal09} {\sc Hough, J.~B.,  Krishnapur, M., Peres, Y.} and {\sc Vir\'ag, B.} (2009). {\it Zeros of Gaussian analytic functions and determinantal point processes.} American Mathematical Society, Providence, RI.

\bibitem[Jancovici, Lebowitz and Manificat(1993)]{jancovicietal93} {\sc Jancovici, B., Lebowitz, J.~L.} and {\sc Manificat, G.} (1993). Large charge fluctuations in classical Coulomb systems. {\it J. Stat. Phys.,} {\bf 72} no. 3, 773--787.

\bibitem[Johansson(1998)]{johansson98} {\sc Johansson, K.} (1998). On fluctuations of eigenvalues of random Hermitian matrices. {\it Duke Math.
J.,} {\bf 91} no. 1, 151--204.

\bibitem[Johansson and Lambert(2015)]{jl15} {\sc Johansson, K.} and {\sc Lambert, G.} (2015). Gaussian and non-Gaussian fluctuations for mesoscopic linear statistics in determinantal processes. {\it arXiv preprint arXiv:1504.06455.}

\bibitem[Kappeler, Pinn and Wieczerkowski(1991)]{kappeleretal91} {\sc Kappeler, T., Pinn, K.} and {\sc Wieczerkowski, C.} (1991). Renormalization group flow of a hierarchical Sine-Gordon model by partial differential equations. {\it Comm. Math. Phys.,} {\bf 136} no. 2, 357--368.

\bibitem[K\"onig(2005)]{konig05} {\sc K\"onig, W.} (2005). Orthogonal polynomial ensembles in probability theory. {\it Probab. Surv.,} {\bf 2}, 385--447.

\bibitem[Kunz(1974)]{kunz74} {\sc Kunz, H.} (1975). The one-dimensional classical electron gas. {\it Ann. Physics,} {\bf 85}, 303--335.


\bibitem[Lambert, Ledoux and Webb(2017)]{llw17} {\sc Lambert, G., Ledoux, M.} and {\sc Webb, C.} (2017). Stein's method for normal approximation of linear statistics of beta-ensembles. {\it arXiv preprint arXiv:1706.10251.}



\bibitem[Lebl\'e(2015)]{leble15} {\sc Lebl\'e, T.} (2015). Local microscopic behavior for 2D Coulomb gases. {\it arXiv preprint arXiv:1510.01506.}

\bibitem[Lebl\'e and Serfaty(2015)]{lebleserfaty15} {\sc Lebl\'e, T.} and {\sc Serfaty, S.} (2015). Large Deviation Principle for Empirical Fields of Log and Riesz Gases. {\it arXiv preprint arXiv:1502.02970.} To appear in {\it Invent. Math.} 

\bibitem[Lebl\'e and Serfaty(2016)]{lebleserfaty16} {\sc Lebl\'e, T.} and {\sc Serfaty, S.} (2016). Fluctuations of Two-Dimensional Coulomb Gases. {\it arXiv preprint arXiv:1609.08088.}

\bibitem[Lebowitz(1983)]{lebowitz83} {\sc Lebowitz, J.~L.} Charge fluctuations in Coulomb systems. {\it Phys. Rev. A,} {\bf 27} no. 3, 1491--1494.

\bibitem[Lenard(1961)]{lenard61} {\sc Lenard, A.} (1961). Exact statistical mechanics of a one-dimensional system with Coulomb forces. {\it J. Math. Phys.,}  {\bf 2}, 682--693.

\bibitem[Lenard(1963)]{lenard63} {\sc Lenard, A.} (1963). Exact statistical mechanics of a one-dimensional system with Coulomb forces. III.
Statistics of the electric field. {\it J. Math. Phys.,} {\bf 4}, 533--543.


\bibitem[Marchetti and Perez(1989)]{marchettiperez89} {\sc Marchetti, D.~H.~U.} and {\sc Perez, J.~F.} (1989). The Kosterlitz-Thouless phase transition in two-dimensional hierarchical Coulomb gases. {\it J. Stat. Phys.,} {\bf 55} nos. 1-2, 141--156.

\bibitem[Marino, Majumdar, Schehr and Vivo(2014)]{marinoetal14} {\sc Marino, R., Majumdar, S.~N., Schehr, G.} and {\sc Vivo, P.} (2014). Phase transitions and edge scaling of number variance in Gaussian random matrices. {\it Phys. Rev. Lett.,} {\bf 112}, 254101.

\bibitem[Marino, Majumdar, Schehr and Vivo(2016)]{marinoetal16} {\sc Marino, R., Majumdar, S.~N., Schehr, G.} and {\sc Vivo, P.} (2016). Number statistics for $\beta$-ensembles of random matrices: applications to trapped fermions at zero temperature. {\it Phys. Rev. E,} {\bf 94}, 032115.


\bibitem[Martin(1988)]{martin88} {\sc Martin, Ph.} (1988). Sum rules in charged fluids. {\it Rev. Mod. Phys.,} {\bf 60} no. 4, 1075--1127.

\bibitem[Martin and Yalcin(1980)]{martinyalcin80} {\sc Martin, Ph.} and {\sc Yalcin, T.} (1980). The charge fluctuations in classical Coulomb systems. {\it J. Stat. Phys.,} {\bf 22} no. 4, 435--463.


\bibitem[Nazarov and Sodin(2011)]{nazarovsodin11} {\sc Nazarov, F.} and {\sc Sodin, M.} (2011). Fluctuations in random complex zeroes: asymptotic normality revisited. {\it Int. Math. Res. Not. IMRN,} {\bf 2011} no. 24, 5720--5759. 

\bibitem[Nazarov, Sodin and Volberg(2007)]{nsv07} {\sc Nazarov, F., Sodin, M.} and {\sc Volberg, A.} (2007). Transportation to random zeroes by the gradient flow. {\it Geom. Funct. Anal.,} {\bf 17} no. 3, 887--935.

\bibitem[Nazarov, Sodin and Volberg(2008)]{nsv08} {\sc Nazarov, F., Sodin, M.} and {\sc Volberg, A.} (2008). The Jancovici-Lebowitz-Manificat law for large fluctuations of random complex zeroes. {\it Comm. Math. Phys.,} {\bf 284} no. 3, 833--865. 

%\bibitem[Nicol\`o and Perfetti(1989)]{nicoloperfetti89} {\sc Nicol\`o, F.} and {\sc  Perfetti, P.} (1989). The Sine-Gordon field theory model at $\alpha^2 = 8\pi$, the non-superrenormalizable theory. {\it Comm. Math. Phys.,} {\bf  123} no. 3, 425--452.

\bibitem[Pastur(2006)]{pastur06} {\sc Pastur, L.} (2006). Limiting laws of linear eigenvalue statistics for Hermitian matrix models. {\it J. Math. Phys.,}  {\bf 47} no. 10, 103303, 22 pp. 

\bibitem[Peres and Sly(2014)]{peressly14} {\sc Peres, Y.} and {\sc Sly, A.} (2014). Rigidity and tolerance for perturbed lattices. {\it arXiv preprint arXiv:1409.4490.}

\bibitem[Petz and Hiai(1998)]{petzhiai98} {\sc Petz, D.} and {\sc Hiai, F.} (1998). Logarithmic energy as an entropy functional. {\it Contemp. Math.,} {\bf 217}, Amer. Math. Soc., Providence, RI.

\bibitem[Radin(1981)]{radin81} {\sc Radin, C.} (1981). The ground state for soft disks. {\it J. Stat. Phys.,} {\bf 26} no. 2, 365--373.

\bibitem[Rider and Vir\'ag(2007)]{ridervirag07} {\sc Rider, B.} and {\sc Vir\'ag, B.} (2007). The noise in the circular law and the Gaussian free field. {\it Int. Math. Res. Not. IMRN,} {\bf 2007} no. 2, Art. ID rnm006, 33 pp. 

\bibitem[Rougerie and Serfaty(2016)]{rougerieserfaty16} {\sc Rougerie, N.} and {\sc Serfaty, S.} (2016). Higher-dimensional Coulomb gases and renormalized energy functionals. {\it Comm. Pure Appl. Math.,} {\bf 69} no. 3, 519--605.


\bibitem[Sandier and Serfaty(2015)]{sandierserfaty15} {\sc Sandier, E.} and {\sc Serfaty, S.} (2015). 2D Coulomb gases and the renormalized energy. {\it Ann. Probab.,} {\bf 43} no. 4, 2026--2083. 

\bibitem[Serfaty(2014)]{serfaty14} {\sc Serfaty, S.} (2014). Ginzburg-Landau vortices, Coulomb gases, and renormalized energies. {\it J. Stat. Phys.,} {\bf 154} no. 3, 660--680.

\bibitem[Serfaty(2015)]{serfaty15} {\sc Serfaty, S.} (2015). {\it Coulomb gases and Ginzburg-Landau vortices.} European Mathematical Society (EMS), Z\"urich.

\bibitem[Shcherbina(2013)]{shcherbina13} {\sc Shcherbina, M.} (2013). Fluctuations of linear eigenvalue statistics of $\beta$ matrix models in the multi-cut regime. {\it J. Stat. Phys.,} {\bf 151} no. 6, 1004--1034.

\bibitem[Soshnikov(2002)]{soshnikov02} {\sc Soshnikov, A.} (2002). Gaussian limit for determinantal random point fields. {\it Ann. Probab.,} {\bf 30} no. 1, 171--187. 

\bibitem[Tao and Vu(2008)]{taovu08} {\sc Tao, T.}  and {\sc Vu, V.} (2008). Random matrices: the circular law. {\it Commun. Contemp. Math.,} {\bf 10} no. 2, 261--307.

\bibitem[Tao and Vu(2013)]{taovu13} {\sc Tao, T.} and {\sc Vu, V.} (2013). Random matrices: sharp concentration of eigenvalues. {\it Random Matrices Theory Appl.,} {\bf 2} no. 3, 1350007, 31 pp. 

\bibitem[Tao and Vu(2015)]{taovu15} {\sc Tao, T.} and {\sc Vu, V.} (2015).  Random matrices: universality of local spectral statistics of non-Hermitian matrices. {\it Ann. Probab.,} {\bf 43} no. 2, 782--874. 

\bibitem[Torquato, Scardicchio, and Zachary(2008)]{torquatoetal08} {\sc Torquato, S., Scardicchio, A.} and {\sc Zachary, C.~E.} (2008). Point processes in arbitrary dimension from fermionic gases, random matrix theory, and number theory. {\it J. Stat. Mech. Theory Exp.,} {\bf 2008}, no. 11, P11019.

\bibitem[Wieand(2002)]{wieand02} {\sc Wieand, K.} (2002). Eigenvalue distributions of random unitary matrices. {\it Probab. Theory Related Fields,} {\bf 123} no. 2, 202--224. 
\end{thebibliography}
\end{document}